\documentclass{elsarticle}
\usepackage{multirow}
\usepackage{braket,amsfonts,amsopn,bm,geometry}
\usepackage{array}

\usepackage[linesnumbered,ruled,vlined,algo2e]{algorithm2e}
\usepackage{algorithm}
\usepackage{amsthm,amsmath,amssymb}
\usepackage{mathrsfs}

\usepackage{amssymb,amsfonts,amsmath,latexsym,epsfig,euscript,color}
\usepackage[utf8]{inputenc}

\newtheorem{theorem}{Theorem}

\newtheorem{corollary}[theorem]{Corollary}
\newtheorem{proposition}[theorem]{Proposition}

\newcounter{mymac@matlab}
\newcommand{\MATLAB}{{\sc matlab}%
	\ifnum\value{mymac@matlab}<1%
	\textsuperscript{\textregistered}%
	\setcounter{mymac@matlab}{1}%
	\fi%
}

\DeclareMathOperator*{\argmin}{arg\,min}

\begin{document}
	\begin{frontmatter}
		\title{Stein-based preconditioners for weak-constraint 4D-var} 
	
	\author[UniBo]{Davide Palitta}
	\ead{davide.palitta@unibo.it}
	\affiliation[UniBo]{organization={Dipartimento 
			di Matematica and AM$^2$,
			Alma Mater Studiorum - Universit\`a di Bologna}, addressline={Piazza di Porta S. Donato, 5, I-40127}, city={Bologna}, country={Italy}, }

	\author[edi,ox]{Jemima M. Tabeart}
	\ead{tabeart@maths.ox.ac.uk}
	\affiliation[edi]{organization={School of Mathematics, University of Edinburgh}, addressline={Peter Guthrie Tait Road}, city={Edinburgh}, country={UK}}
	\affiliation[ox]{ organization={Now at: Mathematical Institute, University of Oxford}, addressline={Andrew Wiles Building}, city={Oxford}, country={UK}}
		
		
		
		\begin{abstract}
			{Algorithms for data assimilation
				try to predict the most likely state of a dynamical system by combining information from
				observations and prior models. {\color{black} Variational approaches, such as the weak-constraint four-dimensional variational data assimilation formulation considered in this paper, can ultimately be
					interpreted as a minimization problem.}
				One of the main challenges of such a formulation is the solution of large 
				{\color{black} linear systems of equations
					which arise within the inner linear step of the adopted nonlinear solver. Depending on the selected approach, these linear algebraic problems amount to either a saddle point linear system or a symmetric positive definite (SPD) one. Both formulations can be solved by means of a Krylov method, like GMRES or CG, that needs
					to be preconditioned to ensure fast convergence in terms of the number of iterations}.
				In this paper we illustrate novel, efficient preconditioning operators which involve the solution of
				certain Stein matrix equations. In addition to achieving better computational performance,
				the latter machinery allows us to derive tighter bounds for the eigenvalue distribution of
				the preconditioned linear system {\color{black}for certain problem settings}.
				A panel of diverse numerical results displays the effectiveness of the proposed methodology compared to current state-of-the-art approaches.}
		\end{abstract}
		
		\begin{keyword}
			4D-Var \sep data assimilation \sep preconditioning \sep Stein equations   
			
		\end{keyword}
	\end{frontmatter}
	
	\section{Introduction}
	Given a computational model for a dynamical system, data assimilation aims to merge observational, measured data of that system with prior model information to obtain a better estimate of the system state at a specified time. The most mature application of data assimilation is to numerical weather prediction (NWP), where it is used to obtain the initial conditions for forecasts~\cite[]{carrassi2018data}, but in recent years data assimilation approaches have been studied more broadly within earth sciences, ecology, and neuroscience; see, e.g.,~\cite[]{hess-22-2575-2018, nakamura2015inverse, schiff2011neural}.  
	In particular, observations $y_i\in\mathbb{R}^{p_i}$ at time $t_i\in[t_0,t_N]$ are combined with prior information $x_b\in\mathbb{R}^s$ from a model to compute the most likely state $x_i\in\mathbb{R}^s$ of the system at time $t_i$. It is typically assumed that the background state $x_b$ can be written as $x_b=x_0^t+\epsilon^b$ where $x_0^t$ denotes the true initial state of the system with the error $\epsilon^b$ . This error is distributed according to a normal distribution with error covariance matrix $B\in\mathbb{R}^{s\times s}$ and zero mean, i.e., $\epsilon^b\sim\mathcal{N}(0,B)$.  Similarly, we write each observation  in terms of the true initial state as  $y_i=\mathcal{H}_i(x_i^t)+\epsilon_i$ with the observation error $\epsilon_i\sim\mathcal{N}(0,R_i)$ for all $i=0,\ldots,N$. In order to map between observation and state space, we introduce a possibly nonlinear operator $\mathcal{H}_i:\mathbb{R}^s\rightarrow\mathbb{R}^{p_i}$ which connects the true state $x_i^t$ at time $t_i$ and the observational data $y_i$.
	
	One of the notable aspects of  the weak-constraint four-dimensional variational assimilation problem (4D-Var) is the propagation of the computed states. If $x_{i-1}$ denotes the state variable at time $t_{i-1}$, this is propagated to the next observation time $t_i$ via an imperfect forecast model $\mathcal{M}_i$ such that $x_i=\mathcal{M}_i(x_{i-1})+\epsilon_i^m$ where $\epsilon_i^m\sim\mathcal{N}(0,Q_i)$ for all $i=1,\ldots,N$.
	{\color{black}This is probably the main difference between the weak- and strong-constrained 4D-Var approaches. Indeed, in the latter methodology the forecast model is supposed to be exact.}
	
	The ultimate goal of {\color{black}weak-constrained} 4D-Var is then minimizing the following functional 
	\begin{align}\label{eq:objective_fun}
		J(x)=&(x_0-x_b)^TB^{-1}(x_0-x_b)+\sum_{i=0}^N(y_i-\mathcal{H}_i(x_i))^TR_i^{-1}(y_i-\mathcal{H}_i(x_i))\notag\\
		&+\sum_{i=1}^N(x_i-\mathcal{M}_i(x_{i-1}))^TQ_i^{-1}(x_i-\mathcal{M}_i(x_{i-1})),
	\end{align}
	where $x=(x_0^T,\ldots,x_N^T)^T\in\mathbb{R}^{(N+1)s}$ collects all the state variables $x_0,\ldots,x_N$. 
	
	
	The objective function~\eqref{eq:objective_fun} is often minimized by means of an incremental approach~\cite[]{Courtieretal1994}. Roughly speaking, this consists of a Gauss-Newton scheme where at each iteration a linearised problem needs to be solved. It has been shown that such a linear, inner problem can be reformulated as a large, sparse, symmetric, but also very structured saddle point linear system; see, e.g., \cite{fisher2017parallelization,fisher2018low,green2019model}. {\color{black}A more traditional approach consists of solving the SPD linear system stemming from the quadratic optimization problems arising from the adopted Gauss-Newton procedure; see e.g. \cite{el2015conditioning,tr2006accounting}.
	}
	
	
	Krylov methods like the Generalized Minimal RESidual method (GMRES)~\cite[]{SaadSchultz86} and the MINimal RESidual method (MINRES) \cite[]{PaigeSaunders75} are powerful tools for the solution of saddle point linear systems. See, e.g., the survey paper~\cite{benzi_golub_liesen_2005}. {\color{black} Similarly, the Conjugate Gradient method (CG)~\cite{HestenesStiefel1952} is the most commonly used solver for SPD linear systems.
		In both scenarios,} it is vital to choose good preconditioners for the adopted iterative scheme to ensure fast convergence in terms of both the number of iterations and wallclock times.
	
	A variety of preconditioners for the saddle point and SPD 4D-Var problems have been proposed in the literature; see, e.g., \cite{FreitagGreen18,GrattonEtal18,tabeart2021saddle,dauvzickaite2021randomised,tshimanga2008limited}.
	{\color{black}While these operators enjoy some appealing features, e.g., they guarantee parallisability in the saddle-point context}, they also neglect important features of the original linear system to achieve affordable computational costs. This worsens the capability of the preconditioners to reduce the overall iteration count. In the operational NWP setting this is particularly problematic, as in practice the maximum number of iterations is capped by a very small number compared to the dimension of the problem \cite[]{carrassi2018data}. Therefore, any preconditioning method that reduces the iteration count without dramatically increasing the computational cost is likely to be highly beneficial.

	In this work we propose to fully exploit the inherent block structure of {\color{black} both the saddle point and SPD formulations} within a matrix-oriented GMRES/CG approach, see, e.g., ~\cite{FreitagGreen18,StollBreiten15,Kressetal2014}. Such machinery naturally leads to the design of more efficient preconditioning operators with Kronecker structure. These new preconditioners yield beneficial theoretical properties, thus achieving faster convergence in terms of number of iterations than state-of-the-art approaches, while maintaining a low computational cost and an easy-to-parallelize nature. 
	
	The framework proposed in this paper requires moderate values of $p$ and $s$ (e.g. $\mathcal{O}(10^3)$) to be computationally successful. We must mention that these restrictions on the problem dimensions may be unrealistic for NWP applications but can be reasonable for other data assimilation problems such as {\color{black}parameter estimation tasks for low-dimensional parameter domains see, e.g.,~\cite{application1,application2} for some examples in agriculture. The fresh methodology we present in this paper can also be successfully applied when data assimilation is combined with model order reduction. This interesting scenario sees a first reduction step aimed at reducing the state dimension. Then the reduced model is utilized within the selected data assimilation approach; see, e.g.,~\cite{application3}. Weak-constraint 4D-Var may be a particularly appropriate choice of data assimilation scheme to be combined with model order reduction, as it is able to take the additional model error coming from the reduction step into account explicitly. 
	}
	Moreover, the techniques we develop here serve as the initial step towards novel procedures tackling large-scale problems such as the ones stemming from NWP. This will be the subject of future work.

	{\color{black} We additionally note that the weak-constraint 4D-Var problem requires more computational resource than the strong-constraint formulation, in addition to the need to prescribe model error covariance matrices. Alternative approaches have been proposed which incorporate some model error information within a strong-constraint 4D-Var approach via an inflated covariance approach, e.g. \cite{howes2017accounting,gejadze2017implicit}. However, as these methods both require the inversion of the inflated observation error covariance term, users are limited to the use of observation error covariance matrices that are easy to invert. Our new approach allows full flexibility for all data assimilation parameters, as well as revealing and exploiting the Kronecker  structure that is obscured in the usual primal form. Indeed, the inflated covariance approach is expected to perform poorly for large numbers of observation times, whereas the approach proposed in this work scales very well with $N$.}
	
	{\color{black}
		Here is a synopsis of the paper. In section~\ref{The saddle-point linear system} we recall the formulation of the SPD and saddle point linear systems stemming from weak-constraint 4D-Var. We briefly introduce matrix-oriented GMRES and CG in section~\ref{Matrix-oriented GMRES and MINRES} and in section~\ref{Preconditioning operators} we describe a general preconditioning framework to be embedded in these routines. 
		In section~\ref{The case of different Ms} we address the case of {\color{black} observation-time} dependent $\mathcal{M}_i$ and propose an original, efficient preconditioning operator. The latter is very similar to the original operator with a single exception. In particular, the original forward operator is approximated by a suitable, {\color{black} observation-time} independent one, namely  $\mathcal{M}_i\equiv\widehat{\mathcal{M}}$ for all $i$, in the preconditioning operator. Thanks to this feature, we are able to show that the inversion of a certain matrix $\mathbf{L}$, which is the predominant computational bottleneck of state-of-the-art preconditioning procedures for 4D-Var, is in fact equivalent to solving a Stein matrix equation. 
		In addition to leading to some insights regarding the selection of a suitable $\widehat{\mathcal{M}}$, we describe in section~\ref{More performing Schur-complement approximations} how the matrix-oriented perspective allows the efficient incorporation of information from the observation term within the Schur complement of the saddle point system or, equivalently, in the preconditioner of the SPD problem, by adapting an approach proposed in \cite{tabeart2021lowrank}. We note that the observation term has often been completely neglected in state-of-the-art preconditioners, but can be incorporated approximately within the Kronecker preconditioning framework at a moderate computational cost. 
		In section~\ref{Time-independent M}, some further considerations are given in the case that also the original
		forecast model $\mathcal{M}_i$ is {\color{black} observation-time independent} itself, namely $\mathcal{M}_i\equiv\mathcal{M}$ for all $i=1,\ldots,N$ in the forward operator.} A number of numerical results showing the potential of our fresh, successful strategy are reported in section~\ref{Numerical experiments}. We finish in section~\ref{Conclusions and outlook} by drawing some conclusions and presenting possible outlooks.
	
	Throughout the paper we adopt the following notation. Capital italic letters ($\mathcal{A}$) denote block matrices whose blocks have a Kronecker structure. These blocks, and in general matrices having a Kronecker structure, are denoted by capital bold letters ($\mathbf{A}$) whereas simple capital letters ($A$) are used for general matrices without any Kronecker structure. $I_N$ denotes the identity matrix of dimension $N$. The subscript is omitted whenever the dimension of $I$ is clear from the context. The $i$-th vector of the canonical basis of $\mathbb{R}^N$ is denoted by $e_i$. The Kronecker product is denoted by $\otimes$, whereas $\circ$ represents the Hadamard product. Given a matrix $X\in\mathbb{R}^{n\times n}$, $\text{vec}(X)\in\mathbb{R}^{n^2}$ is the vector collecting the columns of $X$ on top of one another. For instance, the variable $x$ in~\eqref{eq:objective_fun} can be written as $x=\text{vec}([x_0,\ldots,x_N])$. To conclude, $\lambda(A)$ denotes the spectrum of the matrix $A$, with $\lambda_{\max}(A) = \lambda_1(A)\geq\lambda_2(A)\geq \ldots  \geq\lambda_N(A) = \lambda_{\min}(A)$.
	
	\section{Linear system formulations}\label{The saddle-point linear system}
	As previously mentioned, the vector state $x$ which minimizes~\eqref{eq:objective_fun} can be computed by an incremental approach~\cite[]{Courtieretal1994} where the cost function~\eqref{eq:objective_fun} is approximated by a quadratic function of the increment $\delta x^{(\ell)}=x^{(\ell+1)}-x^{(\ell)}$, with $x^{(\ell)}$ being the $\ell$-th Gauss-Newton iterate. 
	If $\delta x=\text{vec}([\delta x_0,\ldots,\delta x_N])$, the
	quadratic objective function is given by
	\begin{align*}
		\delta J^{(\ell)}(\delta x)=&(\delta x_0-b_0^{(\ell)})^TB^{-1}(\delta x_0-b_0^{(\ell)})+\sum_{i=0}^N(d_i^{(\ell)}-H_i^{(\ell)}\delta x_i)^TR_i^{-1}(d_i^{(\ell)}-H_i^{(\ell)}\delta x_i)\\
		&+\sum_{i=1}^{N}(\delta x_{i}-M_i^{(\ell)}\delta x_{i-1})-c_i^{(\ell)})^TQ_i^{-1}(\delta x_{i}-M_i^{(\ell)}\delta x_{i-1}-c_i^{(\ell)})),
	\end{align*}
	where $b_0^{(\ell)}=x_b-x_0^{(\ell)}$, $d_i^{(\ell)}=y_i-\mathcal{H}_i(x_i^{(\ell)})$, $c_i^{(\ell)}=\mathcal{M}_i(x_{i-1}^{(\ell)})-x_i^{(\ell)}$, and
	$H_i^{(\ell)}$, $M_i^{(\ell)}$ are linearizations of $\mathcal{H}_i$ and $\mathcal{M}_i$ about $x_i^{(\ell)}$, respectively. {\color{black}We note that ${B},{Q}_i$ and ${R}_i$ are covariance matrices so that they are symmetric and positive semi-definite by construction. In addition, as inverse covariance matrices are required in the objective function formulation \eqref{eq:objective_fun} we assume these matrices to be strictly positive definite.}
	Therefore, by dropping the Gauss-Newton index $(\ell)$ for better readability and assuming $p_0=\ldots=p_N=p$, {\color{black} CG can be employed to minimize $\delta J$ by solving the following linear system}
	\begin{align}\label{eq:CGproblem}
		\underbrace{(\mathbf{L}^T\mathbf{D}^{-1}\mathbf{L} + \mathbf{H}^T\mathbf{R}^{-1}\mathbf{H})}_{\mathbf{S}}\delta x = \mathbf{D}^{-1}b + \mathbf{L}^T\mathbf{H}^T\mathbf{R}^{-1}d,
	\end{align}
	where $b=\text{vec}([b_0,c_1,\ldots,c_N])\in\mathbb{R}^{(N+1)s}$, $d=\text{vec}([d_0,\ldots,d_N])\in\mathbb{R}^{(N+1)p}$, and $\mathbf{D}, \mathbf{L}\in\mathbb{R}^{(N+1)s\times (N+1)s}$, $\mathbf{R}\in\mathbb{R}^{(N+1)p\times(N+1)p}$, $\mathbf{H}\in\mathbb{R}^{(N+1)p\times(N+1)s}$ are such that
	\begin{align*}
		\mathbf{D}=\begin{pmatrix}
			B & & & \\
			& Q_1 & & \\
			& & \ddots & \\
			& & & Q_{N}\\
		\end{pmatrix},&\quad\mathbf{L}=\begin{pmatrix}
			I & & & \\
			-M_1& I & & \\
			&\ddots & \ddots & \\
			& & -M_N&I\\
		\end{pmatrix},\\
		&\\
		\mathbf{R}=\begin{pmatrix}
			R_0 & & & \\
			& R_1 & & \\
			& & \ddots & \\
			& & & R_{N}\\
		\end{pmatrix},&\quad
		\mathbf{H}=\begin{pmatrix}
			H_0 & & & \\
			& H_1 & & \\
			& & \ddots & \\
			& & & H_{N}\\
		\end{pmatrix}.
	\end{align*}

	As an alternative to the quadratic minimization \eqref{eq:CGproblem}, $\delta x$ can be computed by solving the following saddle point linear system~\cite{fisher2017parallelization} 
	\begin{equation}\label{eq:saddle-point}
		\underbrace{\begin{pmatrix}
				\mathbf{D} & 0 & \mathbf{L}\\
				0 & \mathbf{R} & \mathbf{H}\\
				\mathbf{L}^T & \mathbf{H}^T & 0\\ 
		\end{pmatrix}}_{=:\mathcal{A}}\begin{pmatrix}
			\delta \eta\\
			\delta\lambda\\
			\delta x\\
		\end{pmatrix}
		=
		\begin{pmatrix}
			b\\
			d\\
			0\\
		\end{pmatrix},
	\end{equation}
	{\color{black}
		We note that both \eqref{eq:CGproblem} and \eqref{eq:saddle-point} contain a lot of inherent block structure\footnote{{\color{black}
				Notice that the coefficient matrix in~\eqref{eq:CGproblem} is the Schur complement of $\mathcal{A}$. This motivates the use of $\mathbf{S}$ for the former.}}.  We propose to fully exploit this structure by using matrix implementations of iterative methods and designing preconditioners with explicit Kronecker structure. We illustrate the main concept by considering a data assimilation problem  where the blocks of $\mathcal{A},\mathbf{S}$ and corresponding preconditioners $\mathcal{P}$ have Kronecker structure. This could arise naturally via consistent observation networks, with fixed observation and model error statistics, at each observation time. 
		In the case that  $Q_1=\cdots=Q_N\equiv Q$, $R_0=\cdots=R_N\equiv R$, and
		$H_1=\cdots=H_N\equiv H$, we can write the terms above compactly by using the inherent Kronecker structure
		$$\mathbf{D}=e_1e_1^T\otimes B+(I_{N+1}-e_1e_1^T)\otimes Q,
		\quad \mathbf{R}=I_{N+1}\otimes R, \quad \mathbf{H}=I_{N+1}\otimes H.$$
		In the more general setting where the covariance matrices and linearised observation operator differ at each time, preconditioners with Kronecker structure can be used within the same setting. We expect the strategy presented in section~\ref{The case of different Ms} to be effective also in the case where we relax the Kronecker assumptions on $\mathbf{R},$ $\mathbf{H}$ and $\mathbf{D}$. This will be the subject of future work.  }
	
	
		
		\subsection{Matrix-oriented GMRES and CG}\label{Matrix-oriented GMRES and MINRES}
		The Kronecker form of $\mathbf{S}$ and the blocks of the coefficient matrix $\mathcal{A}$ naturally suggests the use of matrix-oriented Krylov subspace methods to solve the linear systems~\eqref{eq:CGproblem} and~\eqref{eq:saddle-point}. 
		Depending on the adopted preconditioning operator (see section~\ref{Preconditioning operators}),
		the most popular solution schemes for solving~\eqref{eq:saddle-point} is GMRES, or MINRES if symmetry is preserved. Similarly, CG is employed for~\eqref{eq:CGproblem}.
		It is well-known that the original vector form of such methods can be easily transformed in order to obtain a matrix-oriented formulation of these routines. These implementations can be obtained by exploiting the properties of the Kronecker product~\cite[]{VanLoan00} and the equivalence between the 2-norm of vectors and the matrix inner product, namely $\text{vec}(A)^T\text{vec}(B)=\text{trace}((AB)^TAB)$.
		See, e.g., \cite{FreitagGreen18,StollBreiten15,PalittaK21,Kressetal2014} and Appendix B.
		
		We would like to point out that none of the Krylov routines used to obtain the results in section~\ref{Numerical experiments} are equipped with low-rank truncations as suggested in~\cite{FreitagGreen18,StollBreiten15,Kressetal2014}. These truncations steps are essential to obtain a feasibile storage demand when very large dimensional problems are considered. Here we suppose $p$, $s$, and $N$ to be moderate, say $\mathcal{O}(10^3)$, so that issues related to the memory consumption in our solvers do not occur in general. Avoiding the employment of any low-rank truncation will be also crucial to obtain the results stated in Proposition~\ref{prop:Mi not equal M} and section~\ref{Spectral results}.  
		\subsection{Preconditioning operators}\label{Preconditioning operators}
		It is well-known that Krylov subspace techniques require  effective preconditioning operators to obtain fast convergence in terms of the number of iterations.

		In the 4D-Var context, many authors considered preconditioners for~\eqref{eq:CGproblem} of the form
		\begin{equation}\label{eq:Shat}
			\mathbf{\widehat S}:=\mathbf{L}^T\mathbf{D}^{-1}\mathbf{L},
		\end{equation}
		which neglect the second term $\mathbf{H}^T\mathbf{R}^{-1}\mathbf{H}$ in the definition of $\mathbf{S}$. See, e.g., \cite{FreitagGreen18,GrattonEtal18,tabeart2021saddle}.  This leads to an easier-to-invert preconditioning operator\footnote{{\color{black} Notice that, due to its structure, $\mathbf{L}$ is always nonsingular, regardless of the $M_i$'s.}}
		as $\mathbf{\widehat S}^{-1}=\mathbf{L}^{-1}\mathbf{D}\mathbf{L}^{-T}$. 
		A key limitation of this preconditioner is that computation of the inverse operators $\mathbf{L}^{-1}$ and $\mathbf{L}^{-T}$ requires many serial matrix products and is thus not parallelisable. One of the main strategies to overcome this issue is the introduction of a further layer of approximation related to employing an operator $\mathbf{\widehat L}\approx \mathbf{L}$ in the definition of $\mathbf{\widehat S}$, such that multiplication of a vector by  $\mathbf{\widehat S}^{-1}=\mathbf{\widehat L}^{-1}\mathbf{D}\mathbf{\widehat L}^{-T}$ can be distributed over multiple processors. Different options for the selection of $\mathbf{\widehat L}$ can be found in, e.g., \cite{fisher2017parallelization,GrattonEtal18,FreitagGreen18,tabeart2021saddle}. 
		
		For saddle-point linear systems of the form~\eqref{eq:saddle-point}, some of the most commonly-used preconditioners are the \emph{block diagonal} and \emph{block triangular} preconditioners. See, e.g., \cite{BenziWathen2008,MurphyEtal00,GrattonEtal18,FreitagGreen18}. In particular, the block diagonal preconditioner is defined as follows
		\begin{equation}\label{eq:def_P_D}
			\mathcal{P}_{\mathcal{D}}:=\begin{pmatrix}
				\mathbf{D} & & \\
				& \mathbf{R} & \\
				& & \mathbf{\widehat S}\\
			\end{pmatrix},
		\end{equation}
		where $\mathbf{\widehat S}$ is again such that $\mathbf{\widehat S}\approx\mathbf{S}=\mathbf{L}^T\mathbf{D}^{-1}\mathbf{L}+\mathbf{H}^T\mathbf{R}^{-1}\mathbf{H}$ is an approximation to the Schur complement $\mathbf{S}$ of $\mathcal{A}$, {\color{black} and it is often of the form~\eqref{eq:Shat}}. 
		
		Similarly, the  block triangular preconditioner is defined as follows
		\begin{equation}\label{eq:def_P_T}
			\mathcal{P}_{\mathcal{T}}:=\begin{pmatrix}
				\mathbf{D} & 0 &    {\mathbf{L}}\\
				& \mathbf{R} & \mathbf{H}\\
				& & \mathbf{\widehat S}\\
			\end{pmatrix}.
		\end{equation}

		A different class of preconditioning operators for data assimilation problems is given by the \emph{inexact constraint} preconditioner
		\begin{equation}\label{eq:def_P_C}
			\mathcal{P}_{\mathcal{C}}:=\begin{pmatrix}
				\mathbf{D} & 0 & \mathbf{\widehat L}\\
				0 & \mathbf{R} & 0\\
				\mathbf{\widehat L}^T &0 & 0\\
			\end{pmatrix},
		\end{equation}
		which does not involve the inexact Schur complement $\mathbf{\widehat S}$.
		
		Clearly, the effectiveness of the preconditioning operators $\widehat{\mathbf S}$, $\mathcal{P}_\mathcal{D}$, $\mathcal{P}_\mathcal{T}$, and $\mathcal{P}_\mathcal{C}$ significantly depends on the adopted approximations $\mathbf{\widehat L}$ and $\mathbf{\widehat S}$.
		In this paper we introduce novel tools which allow for more successful selections of $\mathbf{\widehat L}$ and $\mathbf{\widehat S}$. {\color{black} In particular, in section~\ref{On the inversion of the Stein operator} we propose  a novel approximation $\mathbf{\widehat L}\approx \mathbf{L}$ which amounts to a Stein operator. We will show that the inversion of such $\mathbf{\widehat L}$ is still computationally affordable by exploiting its matrix equation structure, while it leads to a dramatic decrease in the iteration count.} Moreover, we extend this matrix-oriented method to preconditioning operators $\mathbf{\widehat S}$ which explicitly take into account  information from the observation term $\mathbf{H}^T\mathbf{R}^{-1}\mathbf{H}$ of the Schur complement $\mathbf{S}$ (see section~\ref{More performing Schur-complement approximations}) by adapting a low-rank correction approach that was proposed in \cite{tabeart2021lowrank}. The original techniques proposed in this paper lead to the design of preconditioning operators with better theoretical properties (section~\ref{Spectral results}) and more competitive computational records (section~\ref{Numerical experiments}).
		
		We conclude this section by presenting a novel result related to the use of $\mathcal{P}_\mathcal{D}$ in our setting. We report the proof of the following theorem in Appendix A.
		\begin{theorem}\label{th:P_Dbasisvectors}
			If $\mathbf{b}=(b^T,d^T,0)^T$ denotes the right-hand side in~\eqref{eq:saddle-point}, then the orthonormal basis vectors $\{v_1,\ldots,v_m\}$ of the Krylov subspace $K_m(\mathcal{AP}_{\mathcal{D}}^{-1},\mathbf{b})=\text{span}\{\mathbf{b},\mathcal{AP}_{\mathcal{D}}^{-1}\mathbf{b},\ldots,(\mathcal{AP}_{\mathcal{D}}^{-1})^{m-1}\mathbf{b}\}$ computed by GMRES are such that
			$$
			v_{2k-1}=\begin{bmatrix}
				u_{2k-1}\\
				w_{2k-1}\\
				0
			\end{bmatrix}, \quad\text{and}\quad 
			v_{2k}=\begin{bmatrix}
				0\\
				0\\
				z_{2k}
			\end{bmatrix},\quad\text{for any}\;k\geq1.
			$$
		\end{theorem}
		The zero block structure of the basis vectors illustrated in Theorem~\ref{th:P_Dbasisvectors} can be exploited to design more efficient implementations of the preconditioning step involving $\mathcal{P}_\mathcal{D}$. For instance, we can invert the (inexact) Schur complement $\widehat{\mathbf{S}}$ only for alternate iterations. Similarly, the linear systems with $\mathbf{D}$ and $\mathbf{R}$ play a role only in case of an odd iteration index.
		The GMRES orthogonalization step can also benefit from Theorem~\ref{th:P_Dbasisvectors}, as there is no need to explicitly perform the orthonormalization of the blocks which necessarly have to be zero in the current iteration.

		We take advantage of these observations to obtain all the results related to the performance achieved by $\mathcal{P}_\mathcal{D}$, which are reported in section~\ref{Numerical experiments}.
		
		
		\section{A new preconditioning operator}\label{The case of different Ms}
		{\color{black}
			In this section we present the main contribution of this paper. In particular, we propose to use the following operator 
			\begin{equation}\label{eq:Lhat}
				\widehat{\mathbf{L}} = I_{N+1}\otimes I_s - \Sigma\otimes\widehat{M}=\begin{pmatrix}
					I & & & \\
					-\widehat{M}& I & & \\
					&\ddots & \ddots & \\
					& & -\widehat{M}&I\\
				\end{pmatrix},    
			\end{equation}  
			in place of $\mathbf{L}$ within the selected preconditioning framework.
			The matrix $\widehat{M}$ in~\eqref{eq:Lhat}
			is chosen to be some representative value of the $M_i$'s defining $\mathbf{L}$.
		}
		
		In the numerical experiments in section \ref{Numerical experiments} we consider a number of options for $\widehat{M}$ including, one of the $M_i$'s (e.g. the smallest/largest in norm {\color{black}or condition number}, first/last in the sequence), possibly cycling on the index $i$, and the Karcher matrix mean~\cite[]{BiniIann2013} when the $M_i$'s are all SPD\footnote{Alternative matrix means can be used depending on the problem at hand.}.
		
		{\color{black} The employment of the operator $\widehat{\mathbf{L}}$ described in~\eqref{eq:Lhat} in the definition of $\widehat{\mathbf{S}}$, $\mathcal{P}_\mathcal{D}$, $\mathcal{P}_\mathcal{T}$, and $\mathcal{P}_\mathcal{C}$ leads to novel preconditioning operators for~\eqref{eq:CGproblem} and~\eqref{eq:saddle-point} that can significantly outperform other state-of-the-art approaches. 
			In the next proposition we provide some indications of when we might expect using our fresh approach to be particularly effective. To this end, we present theoretical bounds on the eigenvalues of $\mathbf{\widehat{L}}^{-T}\mathbf{L}^T\mathbf{L}\mathbf{\widehat{L}}^{-1}$.}
		
		\begin{proposition}\label{prop:Mi not equal M}
			Let $D_i = \widehat{M}-M_i$ and $\widehat{\mathbf{L}}$ as in~\eqref{eq:Lhat}. The eigenvalues of $\widehat{\mathbf{L}}^{-T}\mathbf{L}^T\mathbf{L}\widehat{\mathbf{L}}^{-1}$ can be bounded above by
			\begin{equation}\label{eq:upperbound_prep}
				1 + \frac{N}{2}\left(\rho_N+\sqrt{\rho_N^2+4\rho_N}\right),
			\end{equation}
			where
			\begin{equation}
				\rho_N = \left\{\begin{array}{ll}
					N\cdot \max_{m=1,\dots,N}\lambda_{max}(D_m^TD_m), &  \text{if}\;\lambda_{\max}(\widehat{M}^T\widehat{M})=1,\\
					&\\
					\frac{1-\lambda_{\max}^N(\widehat{M}^T\widehat{M})}{1-\lambda_{\max}(\widehat{M}^T\widehat{M})}\cdot\max_{m=1,\dots,N}\lambda_{max}(D_m^TD_m),
					& \text{otherwise}.
				\end{array}\right.
			\end{equation}
		\end{proposition}
		
		Due to the multiple levels of approximation used to obtain the result of Proposition~\ref{prop:Mi not equal M}, the bounds are likely to be loose in practice. However, the qualitative information encoded in~\eqref{eq:upperbound_prep} may provide a way to select a `good' choice of $\widehat{M}$, and an indication of when the preconditioner~\eqref{eq:Lhat} is likely to be effective.
		
		In particular, Proposition~\ref{prop:Mi not equal M} indicates that the best results are likely to be obtained when  $\max_{i,j}\|M_i-M_j\|$ is small. If the difference between the linearised model operators is large then the maximum eigenvalue of the difference terms cannot all be kept small. Similarly, the spectral norm of $\widehat{M}$ itself must also be small. If not, then the sum in~\eqref{eq:mu} will blow up rapidly even for moderate values of $N$. Both of these observations provide insight into a heuristic way to select $\widehat{M}$: begin by choosing $M_i$ with smallest norm. If all the values of $||M_i||$ are similar, then it is likely that the $D_i$ term becomes more important -- we can then choose $\widehat{M}$ to be the value of $M_i$ that minimises the average value of $\|D_iD_i^T\|$. {\color{black} See section~\ref{Numerical experiments} for a panel of diverse numerical experiments displaying such trends.
			
			To obtain a successful preconditioning strategy, operating with $\widehat{\mathbf{L}}$ in~\eqref{eq:Lhat} must not be computationally demanding. In particular, the application of the preconditioners $\widehat{\mathbf{S}}$, $\mathcal{P}_\mathcal{D}$, $\mathcal{P}_\mathcal{T}$, and $\mathcal{P}_\mathcal{C}$ always requires the inversion of 
			$\widehat{\mathbf{L}}$. The efficient computation of  
			$\widehat{\mathbf{L}}^{-1}$ will be the subject of the next section.
		}
		

		
		\subsection{On the inversion of the Stein operator $\widehat{\mathbf{L}}$}\label{On the inversion of the Stein operator}
		{\color{black} Thanks to the properties of the Kronecker product, see, e.g.,~\cite{Simoncini16}, the action of $\widehat{\mathbf{L}}$ in~\eqref{eq:Lhat} on a vector $z=\text{vec}(Z)$ can be written as follows
			$$\widehat{\mathbf{L}}z=\text{vec}(Z-\widehat{M}Z\Sigma^T).$$
			A linear operator of the form 
			$$\begin{array}{rrll}
				\mathfrak{L}:&\mathbb{R}^{s\times (N+1)}&\rightarrow&\mathbb{R}^{s\times (N+1)}\\
				& Z&\mapsto& Z-\widehat{M}Z\Sigma^T\\
			\end{array}
			$$
			is called a Stein operator in the matrix equation literature; see, e.g.,~\cite{Simoncini16}. Therefore, $\widehat{\mathbf{L}}z=\text{vec}(\mathfrak{L}(Z))$. Due to this relation,
			hereafter, with abuse of notation, we say that also $\widehat{\mathbf{L}}$ amounts to a Stein operator. The inversion of $\widehat{\mathbf{L}}$ is thus equivalent to inverting $\mathfrak{L}$, and hence to solving a so-called Stein matrix equation 
			\begin{equation}\label{eq:Stein_eq}
				\text{vec}(Z)=\widehat{\mathbf{L}}^{-1}\text{vec}(V) \quad \Longleftrightarrow\quad 
				Z-\widehat{M}Z\Sigma^T=V.    
			\end{equation}
			Similarly, 
			\begin{equation}\label{eq:TStein_eq}
				\text{vec}(Z)=\widehat{\mathbf{L}}^{-T}\text{vec}(V) \quad \Longleftrightarrow\quad 
				Z-\widehat{M}^TZ\Sigma=V.    
			\end{equation}
		}
		Different numerical methods have been proposed in the literature for the efficient solution of Stein matrix equations. See, e.g., \cite{Barraud77,Jbilou11} and \cite[][Section~6]{Simoncini16}.
		
		In our setting, we need to solve equations~\eqref{eq:Stein_eq}
		and~\eqref{eq:TStein_eq} several times. Indeed, depending on the adopted preconditioning scheme, a couple of Stein equations have to be solved at each GMRES/CG iteration. 
		By fully exploiting the structure of the coefficient matrices defining the Stein equations, we illustrate a novel solution procedure that remarkably reduces the computational cost of the preconditioning steps. Our original scheme requires some minor precomputation that can be performed {\color{black}once prior to the start of the} 
		adopted Krylov iterative scheme.
		
		We first notice that we can write 
		\begin{equation}\label{eq:Sigmaform}
			\Sigma=C-e_1e_{N+1}^T,\qquad C=\begin{pmatrix}
				0 & & & 1\\
				1 & 0 & &\\
				& \ddots& \ddots &  \\
				& & 1 & 0\\ 
			\end{pmatrix}.
		\end{equation}
		Thanks to its circulant structure, $C$ can be cheaply diagonalized by the fast Fourier transform (FFT), namely $C=F^{-1}\Pi F$ where 
		\begin{equation*}
			\Pi=\text{diag}(\pi_1,\ldots,\pi_{N+1}), \quad (\pi_1,\ldots,\pi_{N+1})^T=FCe_1,
		\end{equation*}
		and $F$ denotes the discrete Fourier matrix. The observation in~\eqref{eq:Sigmaform} leads to the following result.
		\begin{proposition}\label{Prop:SteinSol}
			Let $\widehat{M}=T\Lambda T^{-1}$ be the eigendecomposition of $\widehat{M}$ with $\Lambda=\text{diag}(\lambda_1,\ldots,\lambda_s)$ and $P\in\mathbb{C}^{s\times (N+1)}$ be such that $P_{i,j}=1/(1-\lambda_i\pi_j)$. Moreover, let ${\color{black}U}=I+\text{diag}((P(\Lambda Fe_1\circ F^{-T}e_{N+1})))$ where $\circ$ denotes the Hadamard product.
			Then the solution $Z$ to the Stein equation in~\eqref{eq:Stein_eq} can be written as
			\begin{equation}\label{eq:SolStein}
				Z=T(Y-W)F^{-T},    
			\end{equation}
			where 
			$$Y=P\circ(T^{-1}VF^{T}),\quad\text{and}\quad W=P\circ({\color{black}U}^{-1}(\Lambda YF^{-1}e_{N+1})e_1^TF^{T}).
			$$
			Similarly, the solution $Z$ to~\eqref{eq:TStein_eq} is such that
			\begin{equation}\label{eq:SolTStein}
				{\color{black}Z}=T^{-T}(G-H)F,    
			\end{equation}
			where 
			$$G=P\circ(T^{T}{\color{black}V}F^{-1}),\quad\text{and}\quad H=P\circ({\color{black}U}^{-1}(\Lambda GF^{T}e_{1})e_{N+1}^TF^{-1}).
			$$
		\end{proposition}

		The computational cost of the solution of the Stein equations \eqref{eq:Stein_eq}--\eqref{eq:TStein_eq} by \eqref{eq:SolStein}--\eqref{eq:SolTStein} amounts to $\mathcal{O}(s^3(N+1)\log(N+1))$ floating point operations:  the cubic term $s^3$ arises from the eigendecomposition of $\widehat{M}$, {\color{black} while the use of the FFT, namely computing the action of the matrices $F$ and $F^{-1}$ in~\eqref{eq:SolStein}--\eqref{eq:SolTStein}, leads to the polylogarithmic term in $N+1$.}
		Even though the eigendecomposition of $\widehat{M}$ can be computed 
		{\color{black}once, prior to the start of the} Krylov routine,  the approach presented in Proposition~\ref{Prop:SteinSol} requires the matrix $\widehat{M}$ to be of moderate size. On the other hand, by fully exploiting the circulant-plus-low-rank structure of $\Sigma$, we can afford sizable values of $N$. See \cite[Section 5]{Palitta21} for constructions similar to the ones stated in Proposition~\ref{Prop:SteinSol} derived for the solution of certain Sylvester equations.
		
		The procedures for solving this Stein equation and its transpose are summarized in Algorithm~\ref{alg:solveStein} and~\ref{alg:solveTStein}, respectively, and they rely on the results presented in Proposition~\ref{Prop:SteinSol}. Notice that only matrix-matrix multiplications and entry-wise operations are performed in Algorithm~\ref{alg:solveStein} and~\ref{alg:solveTStein} making the preconditioning step easy to parallelize. See, e.g., \cite{Gupta1993}.
		
		\begin{algorithm}[t]
			\DontPrintSemicolon
			\SetKwInOut{Input}{input}\SetKwInOut{Output}{output}
			\Input{$P\in\mathbb{C}^{s\times (N+1)}$ and $\Lambda, T,{\color{black}U}\in\mathbb{C}^{s\times s}$ as in Proposition~\ref{Prop:SteinSol}, $V\in\mathbb{R}^{s\times (N+1)}$.}
			\Output{$Z\in\mathbb{R}^{s\times(N+1)}$ solution to $Z- \widehat{M} Z\Sigma^T=V$.}
			\BlankLine
			Compute $Y=P\circ(T^{-1}VF^T)$\;
			
			Compute $W=P\circ({\color{black}U}^{-1}(\Lambda Y F^{-1}e_{N+1})e_1^TF^T)$\;
			
			Set $Z=T^{-1}(Y-W)F^{-T}$\;
			\caption{Solution of the Stein equation $Z- \widehat{M} Z\Sigma^T=V$.}\label{alg:solveStein}
		\end{algorithm}

		\setcounter{AlgoLine}{0}
		\begin{algorithm}[t]
			\DontPrintSemicolon
			\SetKwInOut{Input}{input}\SetKwInOut{Output}{output}
			\Input{$P\in\mathbb{C}^{s\times (N+1)}$ and $\Lambda, T,{\color{black}U}\in\mathbb{C}^{s\times s}$ as in Proposition~\ref{Prop:SteinSol}, ${\color{black}V\in\mathbb{R}^{s\times (N+1)}}$.}
			\Output{{\color{black}$Z\in\mathbb{R}^{s\times(N+1)}$ solution to $Z- \widehat{M}^T Z\Sigma=V$.}}
			\BlankLine
			Compute $G=P\circ(T^TVF^{-1})$\;
			
			Compute $H=P\circ({\color{black}U}^{-1}(\Lambda G F^{T}e_{1})e_{N+1}^TF^{-1})$\;
			
			Set $Z=T^{-T}(G-H)F$\;
			\caption{Solution of the Stein equation $X- \widehat{M}^T X\Sigma=U$.}\label{alg:solveTStein}
		\end{algorithm}

		\subsection{Influence of Schur complement approximations}\label{More performing Schur-complement approximations}
		
		{\color{black} The quality of the Schur complement approximation $\widehat{\mathbf{S}}$ plays an important role in determining the effectiveness of the preconditioning operators for~\eqref{eq:CGproblem} and~\eqref{eq:saddle-point}}. In this work we make use of two choices of $\widehat{\mathbf{S}}$. We briefly consider classic approximations of the form $\widehat{\mathbf{S}} = \widehat{\mathbf{L}}^\top\mathbf{D}^{-1}\widehat{\mathbf{L}}$ as studied in \cite{FreitagGreen18,GrattonEtal18} but involving the new approach for $\widehat{\mathbf L}$ illustrated in the previous section. The second option is motivated by a low-rank approximation proposed in \cite{tabeart2021lowrank} and includes information from the observation term explicitly. If
		
			\begin{equation}\label{eq:Shat r=0}
				\widehat{\mathbf{S}} = {\mathbf{L}}^\top\mathbf{D}^{-1}{\mathbf{L}},
			\end{equation} 
			then the eigenvalues of $\widehat{\mathbf{S}}^{-1}\mathbf{S}$ are given by $(N+1)(s-p)$ unit eigenvalues, and $p(N+1)$ eigenvalues given by $1+\lambda({\mathbf{L}}^{-1}\mathbf{D{L}}^{-T}\mathbf{H}^T\mathbf{R}^{-1}\mathbf{H})$. We note that ${\mathbf{L}}^{-1}\mathbf{D{L}}^{-T}\mathbf{H}^T\mathbf{R}^{-1}\mathbf{H}$ is of rank $p(N+1)$ with non-negative eigenvalues, meaning that the minimum eigenvalue of  $\widehat{\mathbf{S}}^{-1}\mathbf{S}$ is 1. However, the remaining non-unit eigenvalues can be large, for example in the case that $\mathbf{R}$ is ill-conditioned; see, e.g.,~\cite{tabeart2020conditioning}.  
		An alternative choice of $\widehat{\mathbf{S}}$ which allows any extreme eigenvalues arising from the observation term to be accounted for within the preconditioner, comes from considering a low-rank update to \eqref{eq:Shat r=0} of the form
		
		\begin{equation}\label{eq:Shat_exp}
			\widehat{\mathbf{S}} = {\mathbf{L}}^\top\mathbf{D}^{-1}{\mathbf{L}}+ \mathbf{K}_r\mathbf{K}_r^T,
		\end{equation}
		where $\mathbf{K}_r=I_{N+1}\otimes V_r\Upsilon_r^{1/2} \in \mathbb{R}^{(N+1)s\times (N+1)r}$ and $V_r\Upsilon_r^{1/2}\in \mathbb{R}^{s\times r}$ is constructed from the leading $r$ terms of the eigendecomposition  $V\Upsilon V^T = H^T R^{-1}H$. This approach has been studied theoretically in \cite{tabeart2021lowrank}, where it was proved that in addition to increasing the number of unit eigenvalues, increasing $r$ reduces the largest eigenvalues of $\widehat{\mathbf{S}}^{-1}\mathbf{S}$. 
		
		A similar low-rank update approach can also be considered when using an approximation $\widehat{\mathbf{L}}$ to $\mathbf{L}$. In this setting the smallest eigenvalue of $\widehat{\mathbf S}^{-1}\mathbf{S}$ can now be smaller than one, and including more information from the observation term is not guaranteed to reduce bounds on the largest eigenvalue of the preconditioned system. However, this approach has been found to perform well for a number of problems, particularly where $\widehat{\mathbf{L}}$ is a spectrally good approximation to $\mathbf{L}$. In what follows we apply the low-rank update to an approximate first term.

		In a true low-rank approach ($r<<p$) computational efficiency is ensured by applying the Woodbury identity. This avoids applying the inverse of $\mathbf{D}$, which is expensive and allows the re-use of parallelisable or inexpensive approximations to $\mathbf{\widehat{L}}^{-1}$.  As the setting considered in this paper requires blocks that are not too large ($\mathcal{O}(10^3)$), it is not unreasonable to compute a full decomposition of $H^TR^{-1}H$. We therefore propose using the low-rank approach and Woodbury implementation with large values of $r\leq p$, i.e.,
		\begin{equation}\label{eq:Woodbury expression}
			\widehat{\mathbf{S}}^{-1} =  \widehat{\mathbf{L}}^{-1}\mathbf{D}\widehat{\mathbf{L}}^{-T} - \widehat{\mathbf{L}}^{-1}\mathbf{D}\widehat{\mathbf{L}}^{-T}\mathbf{K}_r\left(I_{r(N+1)}+ \mathbf{K}_r^T\widehat{\mathbf{L}}^{-1}\mathbf{D}\widehat{\mathbf{L}}^{-T}\mathbf{K}_r\right)^{-1}\mathbf{K}_r^T\widehat{\mathbf{L}}^{-1}\mathbf{D}\widehat{\mathbf{L}}^{-T}.
		\end{equation}
		
		
		\subsubsection{Algorithmic considerations}\label{Algorithmic considerations}
		
		While the use of {\color{black} the preconditioner }$\widehat{\mathbf{S}}^{-1}$ in \eqref{eq:Woodbury expression} leads to great gains in the convergence properties of the selected preconditioned iterative scheme, especially for $r\approx p$ {\color{black} -- see section~\ref{Numerical experiments}~--} it also poses some computational challenges. {\color{black}We now demonstrate how to implement \eqref{eq:Woodbury expression} in a feasible way by exploiting the Kronecker structure of the new preconditioner.}
		
		{\color{black}One benefit of using the Woodbury formulation is that the efficient implementations of  $\widehat{\mathbf{L}}^{-1}\text{vec}(V)$ and $\widehat{\mathbf{L}}^{-T}\text{vec}(Z)$ that were introduced in section~\ref{On the inversion of the Stein operator} can be reused to apply $\hat{\mathbf{L}}^{-1}\mathbf D \hat{\mathbf{L}}^{-T}$. Similarly, we can exploit the Kronecker structure of $\mathbf{K}_r=I_{N+1}\otimes V_r\Upsilon_r^{1/2}$ to cheaply apply this operator and its transpose.  } 
		
		
		Therefore the main computational bottleneck of the relation \eqref{eq:Woodbury expression} is the solution of the $r(N+1)\times r(N+1)$ linear system with $I_{r(N+1)}+ \mathbf{K}_r^T\widehat{\mathbf{L}}^{-1}\mathbf{D}\widehat{\mathbf{L}}^{-T}\mathbf{K}_r$. {\color{black} We obtain computational gains by first transforming this problem into an equivalent one, and then solving the transformed problem iteratively using an inner matrix-oriented CG problem.}

		{\color{black}\textbf{Solving a transformed problem}:\\ We can make considerable computational savings  by solving a transformed problem that exploits the identity plus Kronecker structure of $\hat{\mathbf{L}}$. Writing  }
			$\widehat M=T\Lambda T^{-1}$, we define
			$$\widetilde{\mathbf{L}}=I_{N+1}\otimes I_s-\Sigma\otimes\Lambda,\quad \widetilde{\mathbf{S}}= \widetilde{\mathbf{L}}^\top\widetilde{\mathbf{D}}^{-1}\widetilde{\mathbf{L}}+ \widetilde{\mathbf{K}}_r\widetilde{\mathbf{K}}_r^T,
			$$
			where $\mathbf{\widetilde D}=e_1e_1^T\otimes T^{-1}BT^{-T}+(I_{N+1}-e_1e_1^T)\otimes T^{-1}QT^{-T}$,
			and $\mathbf{\widetilde K}_r=I_{N+1}\otimes V_r\Upsilon_r^{1/2}$ with $V_r$, $\Upsilon_r$ coming now from the eigendecomposition of $\widetilde H^TR^{-1}\widetilde H$, $\widetilde H=HT$. The computation of $\widetilde{\mathbf{S}}^{-1}$ now involves the solution of a linear system with $I_{r(N+1)}+ \mathbf{\widetilde K}_r^T\mathbf{\widetilde L}^{-1}\mathbf{\widetilde D}\mathbf{\widetilde L}^{-T}\mathbf{\widetilde K}_r$ whose action is performed by following Algorithm~\ref{alg:computeaction}. Notice that the cost of the latter algorithm is now linear in $s$ and polylogarithmic in $N+1$ thanks to the semidiagonalization of $\widehat{\mathbf{L}}$.

			{\color{black}
				For the SPD problem~\eqref{eq:CGproblem}, we apply the preconditioner $\widehat{\mathbf{S}}=(I_{N+1}\otimes T^{-T}) \widetilde{\mathbf{S}}(I_{N+1}\otimes T^{-1})$ as 
				\begin{equation}\label{eq: transformed problem}
					\widehat{\mathbf{S}}^{-1}\text{vec}(V)=(I_{N+1}\otimes T) (\widetilde{\mathbf{S}}^{-1}\text{vec}(T^TV)).    
				\end{equation}
				
				{\color{black}We note the equality here, and that the only assumption required is that the full eigendecomposition of $\widehat{M}$ is available.}

				Similarly, for the saddle-point linear system~\eqref{eq:saddle-point}, we still write $\mathcal{P}_{\mathcal{D}}=\mathcal{T}\mathcal{\widetilde P}_{\mathcal{D}}\mathcal{T}^T$ and $\mathcal{P}_{\mathcal{T}}=\mathcal{T}\mathcal{\widetilde P}_{\mathcal{T}}\mathcal{T}^T$ where 
				$$
				\mathcal{\widetilde P}_{\mathcal{D}}=\begin{pmatrix}
					\mathbf{\widetilde D} & & \\
					& \mathbf{R} & \\
					& & \mathbf{\widetilde S}\\
				\end{pmatrix}, \qquad
				\mathcal{\widetilde P}_{\mathcal{T}}=\begin{pmatrix}
					\mathbf{\widetilde D} & 0& \mathbf{\widetilde L}\\
					& \mathbf{R} & \mathbf{\widetilde H}\\
					& & \mathbf{\widetilde S}\\
				\end{pmatrix},
				$$
				and
				\begin{equation}\label{eq:T_transformation}
					\mathcal{T}=\begin{pmatrix}
						I_{N+1}\otimes T & & \\
						& I_{N+1}\otimes I_p  & \\
						&& I_{N+1}\otimes T^{-T}\\
					\end{pmatrix}.
				\end{equation}
				
				At the $j$th GMRES iteration we perform 
				\begin{equation}\label{eq:prec_differentMs}
					\mathcal{P}_{\mathcal{D}}^{-1}\text{vec}(V)=\mathcal{T}^{-T}(\mathcal{\widetilde P}_{\mathcal{D}}^{-1}(\mathcal{T}^{-T}\text{vec}(V))),
				\end{equation}
				and $\mathcal{\widetilde P}_{\mathcal{D}}^{-1}$ is computed by following the strategy presented in the previous sections. Notice that $\mathcal{T}$ is block diagonal with blocks having a Kronecker form. This rich structure can be exploited to cheaply perform the transformations involving $\mathcal{T}$ itself.
				The same approach is adopted for the block triangular preconditioner $\mathcal{P}_{\mathcal{T}}$. The inexact constraint preconditioner $\mathcal{P}_{\mathcal{C}}$ would not benefit from the semi-diagonalization of $\mathbf{\widehat L}$ as its definition does not include the approximate Schur complement $\mathbf{\widehat S}$.  {\color{black} We note at this stage that the transformed preconditioner is equivalent to \eqref{eq:Woodbury expression}.} {

			}
			
			{\color{black} \textbf{Inner matrix-oriented CG:}\\
				It now remains to solve $(I_{r(N+1)}+\mathbf{\widetilde K}_r^T\mathbf{\widetilde L}^{-1}\mathbf{\widetilde D}\mathbf{\widetilde L}^{-T}\mathbf{\widetilde K}_r)^{-1}x$ efficiently. }
			A naive strategy would consist of assembling the coefficient matrix first, by possibly exploiting the Kronecker structure of the involved factors $\mathbf{\widetilde K}_r$, $\widetilde{\mathbf{L}}$, and $\mathbf{\widetilde D}$. To this end, the scheme proposed in \cite{HaoSim2021} could be employed with straightforward modifications. Even though this step can be carried out 
			{\color{black}prior to the }GMRES/CG {\color{black} iterations}, the construction of the dense matrix $\mathbf{\widetilde K}_r^T{\mathbf{\widetilde L}}^{-1}\mathbf{\widetilde D}{\mathbf{\widetilde L}}^{-T}\mathbf{\widetilde K}_r$ requires the solution of $r(N+1)$ matrix equations making this task computationally unaffordable.
			
			We thus pursue a different path. Since $\mathbf{D}$ and $\mathbf{R}$ are SPD by construction 
			$I_{r(N+1)}+ \mathbf{\widetilde K}_r^T\widetilde{\mathbf {L}}^{-1}\mathbf{\widetilde D}\widetilde{\mathbf{L}}^{-T}\mathbf{\widetilde K}_r$ is SPD as well. 
			Moreover, the Kronecker structure of the latter matrix can be exploited to cheaply compute its action {\color{black} as mentioned above. We therefore propose using an iterative method to approximate the solution of $(I_{r(N+1)}+\mathbf{\widetilde K}_r^T\mathbf{\widetilde L}^{-1}\mathbf{\widetilde D}\mathbf{\widetilde L}^{-T}\mathbf{\widetilde K}_r)^{-1}x$ within the preconditioner by means of a matrix-oriented CG method. This iterative method only requires  }
				{only applications of the operator $I_{r(N+1)}+ \mathbf{\widetilde K}_r^T\widetilde{\mathbf{L}}^{-1}\mathbf{\widetilde D}\widetilde{\mathbf{L}}^{-T}\mathbf{\widetilde K}_r$. Therefore, the overall scheme for solving~\eqref{eq:CGproblem}} and~\eqref{eq:saddle-point} can be seen as an inner-outer iteration~\cite[]{SimSzy2002} whenever the approximation~\eqref{eq:Woodbury expression} with $r>0$ is adopted within the selected preconditioning framework. In particular, the outer Krylov routine (GMRES/CG) is preconditioned with a scheme involving a second, inner Krylov method (CG). Notice that the use of CG within the preconditioning step requires the employment of a flexible variant of the outer Krylov method {\color{black} as we are using different approximate preconditioners for each outer iteration}; see~\cite{SimSzy2002}. The matrix-oriented implementation of flexible GMRES and CG can be easily obtained from their standard form~\cite[]{Saad1993,Notay2000}.

				{\color{black}We stress once again that combining the inner matrix-oriented CG method with the semi-diagonalised approach to solve for $\widetilde{\mathbf{S}}^{-1} $}
				significantly lowers the computational cost of the preconditioning step, especially for the case $r>0$. Indeed, in this case, the cost of the CG iterations involved in the computation of $\widetilde{\mathbf{S}}^{-1}$ is linear in $s$ and polylogarithmic in $N$ thanks to the semi-diagonalization of $\mathbf{\widehat{L}}$; see Algorithm~\ref{alg:computeaction}. {\color{black}We note that by using an inner iterative solver we obtain an approximation to the `true' preconditioner $\widetilde{\mathbf S}$. However, our numerical experiments in section~\ref{Numerical experiments} reveal that we can obtain near optimal performance using this nested approach for a reasonable choice of tolerance within the inner matrix-oriented CG problem}. 
				\setcounter{AlgoLine}{0}
				\begin{algorithm}[t]
					\DontPrintSemicolon
					\SetKwInOut{Input}{input}\SetKwInOut{Output}{output}
					\Input{$Z\in\mathbb{C}^{r\times (N+1)}$, $V_r$, $\Upsilon_r\in\mathbb{R}^{s\times r}$, and $\Lambda$, $\widetilde B$, $\widetilde Q\in\mathbb{R}^{s\times s}$.}
					\Output{$X\in\mathbb{R}^{r\times(N+1)}$ such that $\text{vec}(X)=(I_{r(N+1)}+ \mathbf{\widetilde K}_r^T\mathbf{\widetilde L}^{-1}\mathbf{\widetilde D}\mathbf{\widetilde L}^{-T}\mathbf{\widetilde K}_r)\text{vec}(Z)$.}
					\BlankLine
					Solve 
					$$Y-\Lambda Y\Sigma=V_r\Upsilon_rZ$$ 
					by means of Algorithm~\ref{alg:solveTStein} with $T=I$\;
					
					Compute $W=\widetilde BYe_1e_1^T+\widetilde QY(I_{N+1}-e_1e_1^T)$\;
					
					Solve 
					$$U-\Lambda U\Sigma^T=W$$ 
					by means of Algorithm~\ref{alg:solveStein} with $T=I$\;
					
					Set $X=Z+\Upsilon_rV_r^TU$\;
					\caption{Computing the action of $I_{r(N+1)}+ \mathbf{\widetilde K}_r^T\mathbf{\widetilde L}^{-1}\mathbf{\widetilde D}\mathbf{\widetilde L}^{-T}\mathbf{\widetilde K}_r$.}\label{alg:computeaction}
				\end{algorithm}

				As previously mentioned, matrix CG is often equipped with some low-rank truncations to reduce the overall memory demand; see, e.g., \cite{Kressetal2014}.
				However, storage will not be an issue in our context thanks to the modest problem dimensions we consider. Moreover, the introduction of any low-rank truncation would worsen the performance of the preconditioning step in general. Therefore, we perform no low-rank truncations within the inner matrix CG.
				
				We would like to mention that, similarly to the outer Krylov routine, the inner matrix CG can also be preconditioned to achieve a faster convergence in terms of the number of iterations. However, we were not able to design an effective preconditioning operator for $I_{r(N+1)}+ \mathbf{\widetilde K}_r^T\widetilde{\mathbf{L}}^{-1}\mathbf{\widetilde D}\widetilde{\mathbf{L}}^{-T}\mathbf{\widetilde K}_r$ with a reasonable computational cost. A number of natural preconditioning strategies did not yield improvement in terms of computational speed compared to a plain, unpreconditioned matrix CG implementation. We therefore present an unpreconditioned inner matrix CG in all the numerical experiments reported in section~\ref{Numerical experiments}.

				One may also consider performing 
				the FFT transformations involved in the solution of the Stein equations {\color{black}outside the GMRES/CG iteration}, thinking that this would further decrease the computational cost of the preconditioning steps. However, this would also introduce complex arithmetic in the GMRES/CG iteration, increasing the cost of the overall scheme. Moreover, the application of the FFT can be cheaply performed without forming the discrete Fourier matrix $F$. In particular, in all our numerical tests we employed the Matlab {\tt fft} and {\tt ifft} functions. In light of these considerations, we confine the use of FFT to the preconditioning step only. 
						
						\section{{\color{black}Observation-time} independent $\mathcal{M}$}\label{Time-independent M}
						If the model forecast $\mathcal{M}$ {{\color{black} takes a constant value between each observation time, the same is true for its linearization. In this case}, all the matrices $M_i$ in the definition of $\mathbf{L}$ are the same, namely $M=M_i$ for all $i=1,\ldots,N+1$. In this case, the operator 
							\begin{equation}\label{eq:Ltimeindep}
								\mathbf{L} = \begin{pmatrix}
									I & & & \\
									-M& I & & \\
									&\ddots & \ddots & \\
									& & -M&I\\
								\end{pmatrix}=I_{N+1}\otimes I_s - \Sigma\otimes M,    
							\end{equation}
							
							is a Stein operator itself and we can thus use $\widehat{\mathbf{L}}=\mathbf{L}$ in our preconditioning strategy. In this easier setting, the latter choice leads to narrow eigenvalue distributions of the preconditioned coefficient matrices -- see section~\ref{Spectral results} -- while maintaining high computational efficiency. 
							
							In contrast to what happens in the more general case discussed at the end of section~\ref{Algorithmic considerations}, we can now perform the transformation based on the eigenvector matrix $T$ once 
							before the Krylov routine starts, and not every time the preconditioner is applied.
							For instance, if $M=T\Lambda T^{-1}$,~\eqref{eq:CGproblem} can be written as
							\begin{equation}\label{eq:cg_problemII}
								\widetilde{\mathbf{S}}\widetilde{\delta x}= f,
							\end{equation}
							where $\widetilde{\mathbf{S}}=\widetilde{\mathbf{L}}^T\widetilde{\mathbf{D}}^{-1}\widetilde{\mathbf{L}}+\widetilde{\mathbf{H}}^T\mathbf{R}\widetilde{\mathbf{H}}$, $\widetilde{\mathbf{D}}$, $\widetilde{\mathbf{H}})$, and $\widetilde{\mathbf{L}}$ are as in section~\ref{Algorithmic considerations}, and $f= (I\otimes T^T)(\mathbf{D}^{-1}b + \mathbf{L}^T\mathbf{H}^T\mathbf{R}^{-1}d)$.
							Once $\widetilde{\delta x}$ is computed, we retrieve the actual solution by performing $\delta x=(I_{N+1}\otimes T)\widetilde{\delta x}$.
							
							The same approach can be followed for the saddle point linear system~\eqref{eq:saddle-point}.
							We can write 
							\begin{align*}
								\mathcal{A}=&\begin{pmatrix}
									e_1e_1^T\otimes B+(I_{N+1}-e_1e_1^T)\otimes Q & & I_{N+1}\otimes I_s -\Sigma\otimes M \\
									0 & I_{N+1}\otimes R & I_{N+1}\otimes H\\
									I_{N+1}\otimes I_s -\Sigma^T\otimes M^T & I_{N+1}\otimes H^T & 0
								\end{pmatrix}\\
								=&
								\mathcal{T}\underbrace{
									\begin{pmatrix}
										e_1e_1^T\otimes \widetilde B+(I_{N+1}-e_1e_1^T)\otimes \widetilde Q & & I_{N+1}\otimes I_s -\Sigma\otimes \Lambda \\
										0 & I_{N+1}\otimes R & I_{N+1}\otimes \widetilde H\\
										I_{N+1}\otimes I_s -\Sigma^T\otimes \Lambda & I_{N+1}\otimes \widetilde H^T & 0\\
								\end{pmatrix}}_{\mathcal{\widetilde A}}\mathcal{T}^T,
							\end{align*}

							where $\widetilde B$, $\widetilde Q$,
							$\widetilde H$, and $\mathcal{T}$ are as in section~\ref{Algorithmic considerations}.
								In place of~\eqref{eq:saddle-point} we can thus solve the transformed system 
								\begin{equation}\label{eq:saddle_pointII}
									\mathcal{\widetilde A}\begin{pmatrix}
										\widetilde{\delta \eta}\\
										\widetilde{\delta \lambda}\\
										\widetilde{\delta x}\\
									\end{pmatrix} = \begin{pmatrix}
										\widetilde{b}\\
										\widetilde{d}\\
										0\\
									\end{pmatrix},
								\end{equation}
								where 
								$$
								\begin{pmatrix}
									\widetilde{\delta \eta}\\
									\widetilde{\delta \lambda}\\
									\widetilde{\delta x}\\
								\end{pmatrix}=\mathcal{T}^T\begin{pmatrix}
									\delta \eta\\
									\delta \lambda\\
									\delta x\\
								\end{pmatrix},\quad\text{and}\quad 
								\begin{pmatrix}
									\widetilde{b}\\
									\widetilde{d}\\
									0\\
								\end{pmatrix}=\mathcal{T}^{-1}\begin{pmatrix}
									b\\
									d\\
									0\\
								\end{pmatrix}.
								$$
								Once $(\widetilde{\delta \eta},
								\widetilde{\delta \lambda},
								\widetilde{\delta x})^T
								$ is computed, we retrieve the original solution by 
								$(\delta \eta,
								\delta \lambda,
								\delta x)^T=\mathcal{T}^{-T}(\widetilde{\delta \eta},
								\widetilde{\delta \lambda},
								\widetilde{\delta x})^T
								$. 
							}
							
							The preconditioning operators for~\eqref{eq:cg_problemII} and~\eqref{eq:saddle_pointII} can be obtained by mimicking what we presented in the previous sections. The major difference is the cheaper inversion of the Stein operator {\color{black} $\widetilde{\mathbf{L}}=I_{N+1}\otimes I_s -\Sigma\otimes \Lambda$ which is now ``semi''-diagonalized. 
								The cost of computing  $\widetilde{\mathbf{L}}^{-1}$ is thus linear in $s$ and polylogarithmic in $N+1$.
							}
							\subsection{Spectral results}\label{Spectral results}
							In this section we present bounds on the eigenvalues of the preconditioned systems using our new approach when $\widehat{\mathbf{L}} = \mathbf{L}$.
							
							The spectral properties of linearised data assimilation problem \eqref{eq:CGproblem} have been studied in \cite{tabeart2018conditioning} for the unpreconditioned 3D-Var formulation, and in \cite{tabeart2020conditioning} for strong-constraint 4D-Var preconditioned with the exact first term. Bounds on the spectrum of the (preconditioned) Hessian $\mathbf{S}$ for the weak-constraint problem can be obtained using the same theoretical approaches and replacing $\mathbf{B}$ in those bounds with $\mathbf{L}^T\mathbf{D}^{-1}\mathbf{L}$. 
							
							\begin{proposition}
								
								Let $\widehat{\mathbf{S}} = \mathbf{L}^T\mathbf{D}^{-1}\mathbf{L}$. 
								Then $\widehat{\mathbf{S}}^{-1}\mathbf{S} = \mathbf{I}_{s(N+1)} + \mathbf{L}^{-1}\mathbf{D}\mathbf{L}^{-T}\mathbf{H}^T\mathbf{R}^{-1}\mathbf{H}$ and the eigenvalues of $\widehat{\mathbf{S}}^{-1}\mathbf{S}$ are bounded as follows 
								\begin{equation}
									\lambda(\widehat{\mathbf{S}}^{-1}\mathbf{S}) \in \left[ 1,\frac{\lambda_{\max}(\mathbf{H}^T\mathbf{R}^{-1}\mathbf{H})}{\lambda_{\min}(\mathbf{L}^T\mathbf{D}^{-1}\mathbf{L})} \right].
								\end{equation}
								We note that $\widehat{\mathbf{S}}^{-1}\mathbf{S}$ has $(s-p)(N+1)$ unit eigenvalues.
							\end{proposition}
							\begin{proof}
								Apply \cite[Theorem 4]{tabeart2020conditioning} replacing $\mathbf{B}^{-1}$ with $\mathbf{L}^T\mathbf{D}^{-1}\mathbf{L}$.
							\end{proof}
							
							In the case where the low-rank update to the Schur complement preconditioner presented in section~\ref{More performing Schur-complement approximations} is applied with $\widehat{\mathbf{L}}=\mathbf{L}$, a bound on the maximum eigenvalue is controlled by the largest neglected eigenvalue that is not included in the approximation $\mathbf{K}_r$.
							
							\begin{proposition}[\cite{tabeart2021lowrank}]\label{Prop CG bound low rank}
								Let $\widehat{\mathbf{S}} = \mathbf{L}^T\mathbf{D}^{-1}\mathbf{L}+\mathbf{K}_r\mathbf{K}_r^T$ with $\mathbf{K}_r$ defined as in section~\ref{More performing Schur-complement approximations}. 
								Then ${\mathbf{S}}^{-1}\mathbf{S} = \mathbf{I}_{s(N+1)} + (\mathbf{L}^T\mathbf{D}^{-1}\mathbf{L}+\mathbf{K}_r\mathbf{K}_r^T)^{-1}\mathbf{H}^T\mathbf{R}^{-1}\mathbf{H}$ and the eigenvalues of $\widehat{\mathbf{S}}^{-1}\mathbf{S}$ are bounded between 
								\begin{equation}
									\lambda(\widehat{\mathbf{S}}^{-1}\mathbf{S}) \in \left[ 1,\frac{\lambda_{r+1}}{\lambda_{\min}(\mathbf{L}^T\mathbf{D}^{-1}\mathbf{L})} \right].
								\end{equation}
								where $\lambda_{r+1}$ is the $(r+1)$th largest eigenvalue of $\mathbf{H}^T\mathbf{R}^{-1}\mathbf{H}$, i.e. the largest eigenvalue that is neglected by the low-rank approximation $\mathbf{K}_r\mathbf{K}_r^T$ to $\mathbf{H}^T\mathbf{R}^{-1}\mathbf{H}$.
								We note that $\widehat{\mathbf{S}}^{-1}\mathbf{S}$ has $(s+r-p)(N+1)$ unit eigenvalues.
							\end{proposition}
							
							\begin{corollary}
								If $r=p$ then $\widehat{\mathbf{S}}^{-1}\mathbf{S} = \mathbf{I}_n$ and the eigenvalues of the preconditioned system are all units.
							\end{corollary}
							Approximations of $\widehat{\mathbf{S}}$, either by using randomised approximations of $\mathbf{K}_r$ as proposed in \cite{tabeart2021lowrank}, or inner iterative methods to compute~\eqref{eq:Woodbury expression} as proposed in this paper may lead to eigenvalues of the preconditioned system that are smaller than $1$ or larger than the theoretical upper bound given by Proposition~\ref{Prop CG bound low rank}. We will study the performance of these approximate preconditioners in section~\ref{Numerical experiments}. 
							
							The spectral properties of the preconditioned saddle point problem \eqref{eq:saddle-point} have been studied in \cite{green2019model,FreitagGreen18,tabeart2021lowrank,fisher2018low}, although typically by considering approximations  $\widehat{\mathbf{R}},\widehat{\mathbf{D}}$ and $\widehat{\mathbf{L}}$ rather than using the exact forward model matrices. In this work we instead propose using the exact covariance matrices within the preconditioner, i.e. $\widehat{\mathbf R} =\mathbf{R}$ and $\widehat{\mathbf D} =\mathbf{D}$. 
							This is possible due to the exploitation of the Kronecker structure and the use of matrix iterative methods. 
							
							Let $\lambda(\widehat{\mathbf{S}}^{-1}\mathbf{S})\in[\lambda_{\mathbf{S}},\Lambda_{\mathbf{S}}]$. In what follows we consider how each of the three preconditioners for the saddle point problem~\eqref{eq:saddle-point} are affected by the approximation of $\widehat{\mathbf S}$ to $\mathbf{S}$.   
							We now state bounds on the eigenvalues of the preconditioned saddle point system using each of the preconditioners introduced in section \ref{Preconditioning operators} with $\mathbf{\widehat L}=\mathbf{L}$. 
							
							\begin{proposition}
								With the definitions as stated above, the eigenvalues of $\mathcal{P}_{\mathcal{D}}^{-1}\mathcal{A}$ are real, and satisfy:
								
								\begin{align*}
									\lambda(\mathcal{P}_D^{-1}\mathcal{A})\in& \left[\frac{1-\sqrt{1+4\Lambda_\mathbf{S}}}{2},\frac{1-\sqrt{1+{4\lambda_\mathbf{S}}}}{2}\right] 
									\cup\{1\} \cup\left[\frac{1+\sqrt{1+{4\lambda_\mathbf{S}}}}{2},\frac{1+\sqrt{1+4\Lambda_\mathbf{S}}}{2}\right].
								\end{align*}
								
							\end{proposition}
							\begin{proof}
								The result directly comes from \cite[Theorem 4.2.1]{rees2010preconditioning}.
							\end{proof}
							
							We can see that obtaining an improved estimate of the Schur complement (in a spectral sense) will lead to tighter bounds on the eigenvalues of the preconditioned system when using $\mathcal{P}_D^{-1}\mathcal{A}$. 
							\begin{corollary}\label{cor:diag_prec}
								If $\widehat{\mathbf{S}} = \mathbf{S}$ 
								\begin{equation*}
									\lambda(\mathcal{P}_D^{-1}\mathcal{A})\in \left\{ \frac{1-\sqrt{5}}{2} ,1, 
									\frac{1+\sqrt{5}}{2}\right\}.
								\end{equation*}
							\end{corollary}

							
							The quality of the Schur complement approximation also affects 
							the bounds on the eigenvalues of the block triangular preconditioner. 
							
							\begin{proposition}
								With the definitions as stated above, the eigenvalues of $\mathcal{P}_{\mathcal T}^{-1}\mathcal{A}$ are given by 
								$(s+p)(N+1)$ units, and the remaining $s(N+1)$ eigenvalues are given by the eigenvalues of $\mathbf{S}\widehat{\mathbf{S}}^{-1}$. 
							\end{proposition}
							\begin{proof}
								We consider the product 
								\begin{equation}
									\begin{split}
										\mathcal{A}\mathcal{P}_{\mathcal T}^{-1} &= \begin{pmatrix}
											\mathbf{D} &  0 & \mathbf{L}\\ 0 & \mathbf{R} & \mathbf{H} \\ \mathbf{L}^T& \mathbf{H}^T & 0
										\end{pmatrix}
										\begin{pmatrix}
											\mathbf{D}^{-1} &  0 & \mathbf{D}^{-1}\mathbf{L}\mathbf{\widehat{S}}^{-1}\\ 0 & \mathbf{R}^{-1} & \mathbf{R}^{-1}\mathbf{H}\mathbf{\widehat{S}}^{-1} \\ 0& 0 & -\mathbf{S}^{-1}
										\end{pmatrix} \\&= \begin{pmatrix}
											I & 0 & 0 \\ 0 & I & 0\\ \mathbf{L}^T\mathbf{D}^{-1} & \mathbf{H}^T\mathbf{R}^{-1} & \mathbf{S}\mathbf{\widehat{S}}^{-1}
										\end{pmatrix}.
									\end{split}
								\end{equation}
								The eigenvalues of $ \mathcal{A}\mathcal{P}_{\mathcal T}^{-1}$, and by similarity  $\mathcal{P}_{\mathcal T}^{-1}\mathcal{A}$ are therefore given by $1$ and the eigenvalues of $\mathbf{S}\widehat{\mathbf{S}}^{-1}$.
							\end{proof}
							
							\begin{corollary}\label{cor:triang_prec}
								If  $\widehat{\mathbf{S}} = \mathbf{S}$, then 
								$\lambda(\mathcal{P}_{\mathcal T}^{-1}\mathcal{A}) = 1$.
							\end{corollary}

							\begin{proposition}
								With the definitions as stated above, the eigenvalues of $\mathcal{P}_{\mathcal{C}}^{-1}\mathcal{A}$ consist of $(2s-p)(N+1)$ unit eigenvalues, with the remaining $2p(N+1)$ eigenvalues given by 
								\begin{equation}\label{eq:IC eval bounds}
									\lambda(\mathcal{P}_{\mathcal{C}}^{-1}\mathcal{A}) = 1 \pm \sqrt{\lambda_{i}(\mathbf{R}^{-1}\mathbf{HL}^{-1}\mathbf{DL}^{-T}\mathbf{H}^T)}i
								\end{equation}
							\end{proposition}
							\begin{proof}
								The proof comes from~\cite[Appendix A]{fisher2012weak}. 
							\end{proof}
							
							\section{Numerical experiments}\label{Numerical experiments}
							\subsection{Experimental framework}
							In this section we display some numerical results achieved by the novel preconditioning framework we presented in this paper for the two problems of interest~\eqref{eq:CGproblem} and~\eqref{eq:saddle-point}. Our matrix-oriented strategy is compared to state-of-the-art vector-oriented approaches designed for  linear systems stemming from data assimilation problems.  In particular, we consider the scheme  from~\cite{tabeart2021saddle} where a user-specified parameter $k$ defines the approximation $\mathbf{\widehat L}^{-1}$. This is used within the Hessian/Schur complement approximation $\mathbf{\widehat S}=\mathbf{\widehat L}^T\mathbf{D}^{-1}\mathbf{\widehat L}$; see~\cite[][Section 4]{tabeart2021saddle} for further details\footnote{Notice that choosing $\mathbf{\widehat L}=\mathbf{L}$ is equivalent to setting $k=N+1$ in~\cite{tabeart2021saddle}.}. We note that the approach of~\cite{tabeart2021saddle} is designed to increase parallelisability of the application of $\mathbf{\widehat L}$ and hence a preconditioner. All experiments presented in this section are performed in serial; by exploiting parallel architectures we expect to see large decreases in wallclock times for the approach of~\cite{tabeart2021saddle} when $k<<N+1$, which is not the case for our new strategy. 
							
							We also compare matrix-oriented CG with the improved Schur complement (as presented in section~\ref{More performing Schur-complement approximations}) against the limited memory preconditioner (LMP) approach of \cite{dauvzickaite2021randomised}, where an alternative identity plus low-rank preconditioner is applied as a second level preconditioner. This method can only be implemented in the vectorised setting, and requires the use of $\widehat{\mathbf{L}}\equiv\mathbf{L}$. Hence, in its current formulation, the LMP approach cannot be easily parallelised, making comparison of wallclock times with the matrix-oriented approach more meaningful than for the approach of~\cite{tabeart2021saddle}.
							
							As previously mentioned, we employ a matrix-oriented implementation of CG (respectively GMRES) whereas the standard \emph{vector} form of CG (resp. GMRES) is adopted whenever a preconditioning strategy different from the one introduced in this paper is considered. This is mainly due to the possibility of using existing code for the preconditioners in \cite{FreitagGreen18,GrattonEtal18,tabeart2021saddle}. Indeed, these routines have been designed for standard GMRES and CG and not for their matrix-oriented counterpart. Notice, however, that the two GMRES/CG implementations are equivalent in exact arithmetic as we do not perform any low-rank truncation within matrix GMRES/CG. On the other hand, the matrix-oriented form of GMRES/CG may present some computational advantages due to the Kronecker form of the blocks of the coefficient matrix $\mathcal{A}$ in~\eqref{eq:saddle-point} and $\mathbf{S}$ in~\eqref{eq:CGproblem}. See, e.g., Table~\ref{tab:Mat vs vec GMRES exact L}.
							
							In all the results reported here, the algorithms have been stopped as soon as the iterative method (either in matrix or vector form) relative residual norm becomes smaller than~$10^{-8}$. 
							The same threshold has been used for the inner CG relative residual norm when we adopt the strategy presented in section~\ref{Algorithmic considerations}. 
							
							In what follows, the matrix-oriented implementation of CG (resp. GMRES) will be denoted by {\sc matCG} (resp. {\sc matGMRES}) whereas its standard, vector counterpart by {\sc vecCG} (resp {\sc vecGMRES}).
							
							All the experiments have been run using Matlab (version 2022a) on a machine with a 1.8GHz
							Intel quad-core i7 processor with 15GB RAM on an Ubuntu 20.04.2 LTS operating system.
							
							For both the problem settings we considered, we used data assimilation terms based on those introduced in \cite{tabeart2021saddle}, which we now present briefly. For all experiments the dimension of the state is $s=1000$ and the number of observations at each observation time is given by $p=500$. The background error covariance matrix and model error covariance matrices are produced using an adapted SOAR correlation function~\cite[][Equation (15)]{tabeart2021saddle} with parameters $L_B = 0.6, L_Q= 0.75$, $\sigma_B =0.5$, $\sigma_Q = 0.2$, $100$ non-zero entries per row for $B$ and $120$ for $Q$. The observation error covariance matrix $R \in \mathbb{R}^{500\times 500}$ is produced using the block approach of \cite{tabeart2021saddle}. The observation operator $H \in \mathbb{R}^{500\times 1000}$ has a single unit entry per row, arranged in ascending column order. Each of these terms is repeated in a Kronecker structure to obtain $\mathbf{D} = e_1e^T_1\otimes B + (I_{N+1}-e_1e^T_1 )\otimes Q \in \mathbb{R}^{1000(N+1)\times 1000(N+1)},$ $\mathbf{R}= I_{N+1}\otimes R  \in \mathbb{R}^{500(N+1)\times 500(N+1)}$, and $\mathbf{H} = I_{N+1}\otimes H\in \mathbb{R}^{500(N+1)\times1000(N+1)}
							$. 
							We discuss the two classes of model matrices in the relevant section.
							

							\subsection{Results for Lorenz96}\label{Results for Lorenz 96}
							Our first example is the Lorenz96 problem \cite[]{lorenz1996predictability}, a nonlinear set of coupled ODEs that is often used as a data assimilation test problem due to its chaotic nature. 
							
							Consider $s$ equally spaced points on the unit line, e.g. $\Delta x = \frac{1}{s}$. For $i=1,\dots, s$, we consider
							\begin{equation*}
								\frac{dx_i}{dt} = (x_{i+1}-x_{i-2})x_{i-1} -x_i+8,
							\end{equation*}
							with periodic boundary conditions (i.e. $x_{-1} = x_{s-1}, x_0 = x_s$). We discretize the equations above by using the numerical implementation of \cite{el2015conditioning}, which integrates the model forward in time using a fourth-order Runge-Kutta scheme. We consider $s=1000$, and unless otherwise mentioned, we use $\Delta t = 10^{-6}$.
							
							As illustrated in Proposition~\ref{prop:Mi not equal M}, we expect the strategy proposed in section~\ref{The case of different Ms} to work well whenever the selected $\widehat{M}$ has small spectral norm and is such that the matrices $D_i=\widehat{M}-M_i$ has small spectral norm for all $i=1,\ldots,N$ as well.
							In Table~\ref{tab:my_label2} we report $\|\widehat{M}\|$ and $\max_m\lambda_{\max}(D_m^TD_m)$ for $N=10$ and different choices of $\widehat{M}$ varying $\Delta t$. In particular, we consider as $\widehat{M}$ the symmetrised first and the last matrices in the block subdiagonal of $\mathbf{L}$, namely $Sym(M_1)=1/2(M_1+M_1^T)$ and $Sym(M_{10})=1/2(M_{10}+M_{10}^T)$, and the Karcher mean\footnote{The Karcher mean is computed by means of the routine {\tt positive\_definite\_karcher\_mean} included in the Matlab toolbox Manopt 6.0 \cite[]{manopt}. As optimization procedure, we adopted the Barzilai-Borwein approach presented in~\cite[]{Porcelli} within {\tt positive\_definite\_karcher\_mean}.} of the symmetric parts of the matrices $M_i$, $i=1,\ldots,10$, namely $Sym(M_{i})=1/2(M_i+M_i^T)$, as these are all SPD. In what follows, we denote the Karcher mean by $Karch(M_i)$. We also report the computed upper bound of Proposition~\ref{prop:Mi not equal M}.
							
							For  $\Delta t\le 1\times 10^{-3}$ the aforementioned selections of $\widehat{M}$ lead to very similar results. As $\|\widehat{M}\|$ is close to one and $\max_m\lambda_{\max}(D_m^TD_m)$ is rather small, the upper bound on the eigenvalues of $\widehat{\mathbf{L}}^{-T}\mathbf{L}^T\mathbf{L}\widehat{\mathbf{L}}^{-1}$ is also close to $1$. As $\Delta t$ increases, the norm of the linearised $M_i$ operators move further away from $1$, leading to increases in the upper bound. 
							For the Lorenz96 problem,  $\|M_i\|$ increases monotonically with $i$, meaning that for larger choices of $\Delta t$ the norm of $Sym(M_{10})$ is much larger than $1$. 
							We see that in the case $\Delta t=10^{-1}$ the choice of $\widehat{M}$ makes a large difference to the upper bound in Proposition~\ref{prop:Mi not equal M}. In section~\ref{The case of different Ms}, we proposed to select $\widehat{M}$ with the smallest norm in order to minimise the upper bound on the preconditioned spectrum, or by minimising $\max_m\lambda_{\max}(D_m^TD_m)$ in the case that the norms have similar values. This approach is supported by the results of Table~\ref{tab:my_label2}, and motivates the selection of $\widehat{M}=Sym(M_{1})$ for this problem for the experiments that follow.  
							
							\begin{table}[t]
								\centering
								\begin{tabular}{cc|c c c}
									$\Delta t$ &$\widehat{M}$ &  $Sym(M_{1})$ &  $ Sym(M_{10})$ & $Karch(M_i)$ \\
									\hline
									$1\times 10^{-6}$&$\|\widehat{M}\|$ & $1.000$&  $1.0023$   & $1.0011 $\\
									& $\max_m\lambda_{\max}(D_m^TD_m)$ & $1.7641\times 10^{-5}$   & $1.7641\times 10^{-5}$ &$1.3981\times 10^{-5}$\\
									& Upper bound~\eqref{eq:upperbound_prep} &  $1.0012$&  $1.0012$ & $1.0001$\\
									$1\times 10^{-3}$ &$\|\widehat{M}\|$ & $0.9980$&  $1.0110 $  & $1.0046$ \\
									& $\max_m\lambda_{\max}(D_m^TD_m)$ & $1.401\times 10^{-2}$   & $1.401\times 10^{-2}$ &$1.009\times 10^{-2}$\\
									& Upper bound~\eqref{eq:upperbound_prep}  &  $1.3796$&$1.3912$ & $1.3261$ \\
									$5\times 10^{-2}$ &$\|\widehat{M}\|$ & $0.9980$&  $1.0110 $  & $1.0046$ \\
									& $\max_m\lambda_{\max}(D_m^TD_m)$ & $1.401\times 10^{-2}$   & $1.401\times 10^{-2}$ &$1.009\times 10^{-2}$\\
									& Upper bound~\eqref{eq:upperbound_prep}& $66.9385$ & $11278$& $320.97$ \\
									$1\times 10^{-1}$&$\|\widehat{M}\|$ & $0.8932$&  $8.2359$   & $1.5188$\\
									& $\max_m\lambda_{\max}(D_m^TD_m)$ & $7.9840$   & $10.8191$ &$7.7113$\\
									& Upper bound~\eqref{eq:upperbound_prep} &   $516.7533$& $2.148\times 10 ^{10}$&$8828.74$ \\
								\end{tabular}
								\caption{Example~\ref{Results for Lorenz 96}. $\|\widehat{M}\|$, $\max_m\lambda_{\max}(D_m^TD_m)$ and the upper bound~\eqref{eq:upperbound_prep} of Proposition~\ref{prop:Mi not equal M} for different selection of $\widehat{M}$ and $N=10$.}
							\label{tab:my_label2}
						\end{table}
						
						In Table~\ref{tab:Mat vs vec GMRES Lorenz} we report the performance achieved by our matrix-oriented preconditioning frameworks: $\widehat{\mathbf{S}}$ (with both $r=0$ and $r=p$ in~\eqref{eq:Woodbury expression}) for~\eqref{eq:CGproblem} and $\mathcal{P}_{\mathcal{D}}$, $\mathcal{P}_{\mathcal{T}}$ (with both $r=0$ and $r=p$ in~\eqref{eq:Woodbury expression}), and $\mathcal{P}_{\mathcal{C}}$ for~\eqref{eq:saddle-point}, varying $\widehat{M}$. {\color{black} We compare these results with those attained by the strategy proposed in~\cite{tabeart2021saddle}, where $\widehat{\mathbf L}$ is chosen as follows:
							\begin{equation*}
								\text{the $(i,j)$th block of }\widehat{\mathbf{L}} = \begin{cases}
									\mathbf{I}, &\quad i=j,\\
									-M_i, &\quad i=j \text{ and } i-k\lfloor\frac{i}{k}\rfloor,\\
									0, &\quad \text{ otherwise. }
								\end{cases}
							\end{equation*} 
							
							As this $\widehat{\mathbf L}$ does not have the form of a Stein equation, it is applied using {\sc vecCG/vecGMRES}, with the parameter $k=3.$}  

						\begin{table}[t]
							\centering
							\begin{tabular}{r|c |c c c}
								&$\widehat{\mathbf{S}}$ & $\mathcal{P}_{\mathcal{D}}$ & $\mathcal{P}_{\mathcal{T}}$ & $\mathcal{P}_{\mathcal{C}}$ \\
								\hline
								{\sc matCG/matGMRES}, $\widehat{M} =Sym(M_{1}) $, $r=0$ & 25.7 &  45.0&   29.0&   46.0\\
								{\sc matCG/matGMRES}, $\widehat{M} =Sym(M_{10})$, $r=0$ &25.8 &  45.0 &  29.0&  46.0\\
								{\sc matCG/matGMRES}, $\widehat{M} = karch(M_i)$, $r=0$ &26.3&   45.0&   28.1&   46.0\\
								{\sc matCG/matGMRES}, $\widehat{M} = Sym(M_{1})$, $r=p$ & 3.8 &    7.0&   6.0  & -\\
								{\sc matCG/matGMRES}, $\widehat{M} = Sym(M_{10})$, $r=p$ &    3.2&7.0&  6.0  & -\\
								{\sc matCG/matGMRES}, $\widehat{M} = karch(M_i)$, $r=p$ &  3.5 &    7.0&  6.0  & -\\
								{\sc vecCG/vecGMRES}~\cite[]{tabeart2021saddle}, $k=3$ &   529.5 & 358.0& 188.7&   66.7\\ 
								\hline
								{\sc matCG/matGMRES}, $\widehat{M} = Sym(M_{1})$, $r=0$ & 9.3816  &  7.5924  &  4.9876  &  7.0108\\
								{\sc matCG/matGMRES}, $\widehat{M} = Sym(M_{10})$, $r=0$ &   9.3436  &  7.7803   & 5.0595   & 6.9689\\
								{\sc matCG/matGMRES}, $\widehat{M} = karch(M_i)$, $r=0$ & 9.4084 &   7.6916   & 4.8374   & 6.8653\\
								{\sc matCG/matGMRES}, $\widehat{M} = Sym(M_{1})$, $r=p$ &  3.3120   & 3.2902 &   3.6184  & -\\
								{\sc matCG/matGMRES}, $\widehat{M} = Sym(M_{10})$, $r=p$ &3.0478  &  3.5083   & 3.7205    & -\\
								{\sc matCG/matGMRES}, $\widehat{M} = karch(M_i)$, $r=p$ &   3.4022  &  3.4488 &   3.7182    & -\\
								{\sc vecCG/vecGMRES}~\cite[]{tabeart2021saddle}, $k=3$ &   222.9978  &153.9620 &  91.1391  & 19.7812\\ 
							\end{tabular}
							\caption{Example~\ref{Results for Lorenz 96}. Iterations (top) and wallclock time (bottom) to convergence for the Lorenz96 problem with $\Delta t=10^{-6}$ for $N = 10 $ for different preconditioners. Values are averaged over 10 realisations. The pre-computation for the Karcher mean took 28.7684 seconds.}
							\label{tab:Mat vs vec GMRES Lorenz}
						\end{table}

						From the results in Table~\ref{tab:Mat vs vec GMRES Lorenz} we can see that the use of the $\widehat{\mathbf{L}}$ proposed in section~\ref{The case of different Ms} is very effective in reducing the overall iteration count for all $\widehat{\mathbf{S}}$, $\mathcal{P}_{\mathcal{D}}$, $\mathcal{P}_{\mathcal{T}}$, and $\mathcal{P}_{\mathcal{C}}$. The number of iterations achieved by $\widehat{\mathbf{S}}$, $\mathcal{P}_{\mathcal{D}}$ and $\mathcal{P}_{\mathcal{T}}$ with $r=p$ is remarkably small, especially when compared to the one attained by employing the $\widehat{\mathbf{L}}$ coming from~\cite{tabeart2021saddle} with $k=3$. These small numbers of iterations impact on the wallclock time of the overall solution problem too with $\widehat{\mathbf{S}}$, $r=p$, with $\widehat{M}=Sym(M_{10})$ being the fastest approach we tested. However, we note that the approach of~\cite{tabeart2021saddle} is designed to increase parallelisability of preconditioners, and significant speed up is to be expected for this preconditioner in a parallel setting (all experiments presented here are performed in serial). 
						
						In Table~\ref{tab:Mat vs vec GMRES Lorenz2} we report the results obtained for $N=100$ for a selection of parameters.  For our new strategy, we document the results achieved using $\widehat{M}=Sym(M_{1})$.
						We can see that $\widehat{\mathbf{S}} $, $\mathcal{P}_{\mathcal{D}}$, and $\mathcal{P}_{\mathcal{T}}$ with $r=p$ lead to a very small number of {\sc matCG/matGMRES} iterations also for this problem setting. Moreover, they scale much better than the strategy from~\cite{tabeart2021saddle} in terms of computational timing in this serial setting. However, we recall from Table~\ref{tab:my_label2}  that the norm of the difference between the blocks $M_i$ is small for $\Delta t=10^{-6}$. 
						{\color{black} This is the best scenario for our novel preconditioning approach.}
						
						In Table~\ref{tab:Comp with LMP} we consider the performance of different approximations $\widehat{\mathbf{S}}$ for the SPD problem~\eqref{eq:CGproblem}. In addition to the low-rank approach presented in section~\ref{More performing Schur-complement approximations} we consider the limited memory preconditioning (LMP) approach studied in \cite{dauvzickaite2021randomised}. The LMP approach approximates eigenvalues of $\mathbf{I}+\mathbf{D}^{1/2}\mathbf{L}^{-T}\mathbf{H}^T\mathbf{R}^{-1}\mathbf{H}\mathbf{L}^{-1}\mathbf{D}^{1/2}$, i.e. symmetrically preconditioning with the exact $\mathbf{L}^T\mathbf{D}^{-1}\mathbf{L}$ operator. Spectral information is typically approximated using randomised numerical linear algebra approaches. In addition to requiring spectral information of a much large linear system than the approach considered in section~\ref{More performing Schur-complement approximations}, LMP requires the use of $\widehat{\mathbf{L}} = \mathbf{L}$, meaning that it cannot be readily applied  in the matrix-oriented approach. Hence, computing the full spectrum of $\mathbf{I}+\mathbf{D}^{1/2}\mathbf{L}^{-T}\mathbf{H}^T\mathbf{R}^{-1}\mathbf{H}\mathbf{L}^{-1}\mathbf{D}^{1/2}$ is   prohibitively costly both in terms of storage and computation. {\color{black} We therefore compare {\sc matCG} preconditioned with $\widehat{\mathbf{S}}$ as in~\eqref{eq:Woodbury expression} with $r=0,r=p$ with {\sc vecCG} preconditioned in two different ways. In the first place, we use $\widehat{\mathbf{S}}$ as in~\eqref{eq:Shat_exp} where the inverse of the exact $\mathbf{L}$ is computed by means of the algorithm in~\cite{tabeart2021saddle} by setting $k=N+1$. The low-rank term $\mathbf{K}_r\mathbf{K}_r^T$ is constructed by considering $r=10,$ or $r=p$ eigenpairs of $H^TRH$. The second preconditioning approach for {\sc vecCG} is given by LMP. Also in this case $\mathbf{L}^{-1}$ is computed by following~\cite{tabeart2021saddle} with $k=N+1$. In LMP, the rank of the low-rank approximation to $\mathbf{D}^{1/2}\mathbf{L}^{-T}\mathbf{H}^T\mathbf{R}^{-1}\mathbf{H}\mathbf{L}^{-1}\mathbf{D}^{1/2}$ is set to $10(N+1)$. Notice that this means that we are employing the same number of eigenpairs as  in~\eqref{eq:Shat_exp} for $r=10$. Indeed, if $V_r\Upsilon_rV_r^T\approx H^TRH$, $V_r\in\mathbb{R}^{p\times 10}$, then $\mathbf{K}_r=I_{N+1}\otimes V_r\Upsilon_r^{1/2}$ has rank $10(N+1)$.
						}
						
						\begin{table}[t]
							\centering
							\begin{tabular}{r|c c c c}
								& $\widehat{\mathbf{S}}$ &$\mathcal{P}_{\mathcal{D}}$& $\mathcal{P}_{\mathcal{T}}$ & $\mathcal{P}_{\mathcal{C}}$ \\
								\hline
								{\sc matCG/matGMRES}, $\widehat{M} = Sym(M_1)$, $r=0$ & 213&239 & 173.5 & 242.5\\
								{\sc matCG/matGMRES}, $\widehat{M} = Sym(M_1)$, $r=p$ &9& 15 & 14 & -\\
								{\sc vecCG/vecGMRES}~\cite[]{tabeart2021saddle}, $k=3$&955.4 & -&   938.1  &  570.9\\ 
								\hline
								{\sc matCG/matGMRES}, $\widehat{M} = Sym(M_1)$, $r=0$&827.34 & 499.70 & 356.71 &504.56\\
								{\sc matCG/matGMRES}, $\widehat{M} = Sym(M_1)$, $r=p$ & 94.80& 105.70 & 116.43& -\\
								{\sc vecCG/vecGMRES}~\cite[]{tabeart2021saddle}, $k=3$ & 4757 & 5054 & 5553  &  2243\\ 
							\end{tabular}
							\caption{Example~\ref{Results for Lorenz 96}. Iterations (top) and wallclock time (bottom) to convergence for the Lorenz96 problem with $N = 100 $ for different preconditioners. 10 realisations. {\sc vecGMRES} with $\mathcal{P}_D$ and $\mathbf{\widehat{L}}^{-1}$ computed as in ~\cite{tabeart2021saddle} ($k=3$) did not converge in 1000 iterations.}
							\label{tab:Mat vs vec GMRES Lorenz2}
						\end{table}

						\begin{table}[]
							\centering
							\begin{tabular}{r|cc}
								Preconditioner & Iterations & Wallclock time\\
								\hline
								{\sc matCG}, $\widehat{M} = Sym(M_1)$, $r=0$ & 212.5 & 845.71 \\
								{\sc matCG}, $\widehat{M} = Sym(M_1)$, $r=p$ & 9 & 98.98 \\
								{\sc vecCG}, $k=N+1$, $r=0$& 178.5& 967.73\\
								{\sc vecCG}, $k=N+1$, $r=10$&126.2& 958.20\\
								{\sc vecCG}, $k=N+1$, $r=10(N+1)$, LMP&28&119.22\\
								& 
							\end{tabular}
							\caption{Example~\ref{Results for Lorenz 96}. Iterations (left) and wallclock time (right) to convergence for the Lorenz96 problem with $N = 100 $ for different preconditioners which approximate $\widehat{\mathbf S}$. Values are averaged over 10 realisations. }
							\label{tab:Comp with LMP}
						\end{table}
						
						We see that in terms of iterations the LMP approach results in much larger reductions than the approach of section~\ref{More performing Schur-complement approximations} for the same number of eigenpairs. However, by exploiting the Kronecker structure of the new preconditioning approach, we can incorporate many more terms, leading to very small number of iterations for {\sc matCG} with $r=p$. We also note the improvement in wallclock times when using the matrix-oriented approach, with the fastest times occuring for {\sc matCG} with $r=p$.
						
						\begin{figure}[t]
							\centering
							\includegraphics[width=0.98\textwidth,trim = 40mm 0mm 40mm 0mm, clip]{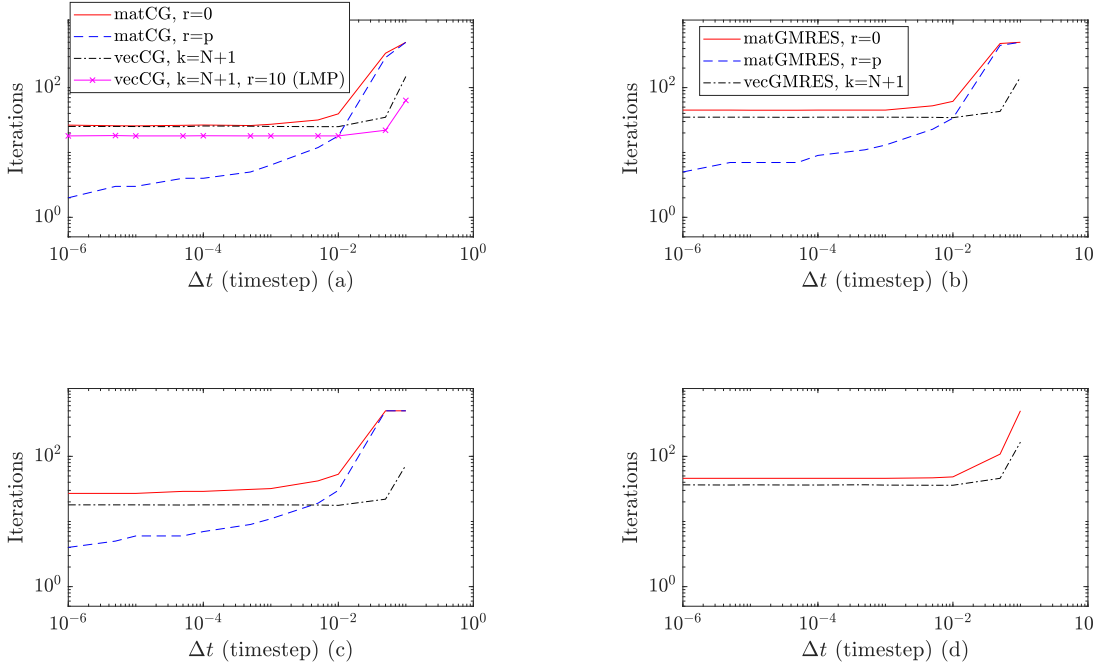}
							\caption{Example~\ref{Results for Lorenz 96}. Iterations to reach convergence with changing time discretization, $\Delta t$, for the Lorenz96 problem for different choices of preconditioner for $N=10$. 
								Panel (a) shows $\widehat{\mathbf{S}}$, (b) shows $\mathcal{P}_\mathcal{D}$, (c)  shows $\mathcal{P}_\mathcal{T}$ and (d) shows $\mathcal{P}_\mathcal{C}$. For all panels the red solid line represents {\sc matGMRES/matCG} with $r=0$ for $\widehat{M} = Sym(M_1)$, blue dashed line represents {\sc matGMRES/matCG} with $r=p$, and black dot-dashed line {\sc vecGMRES/vecCG} for $k=N+1$. For panel (a) the cyan solid line with cross markers shows {\sc vecCG} for $k=N+1$ and $r=10(N+1)$ using the LMP approach. 
								We report averaged behaviour over 10 realisations.}
							\label{fig:Scaling with dt}
						\end{figure}
						
						{\color{black} Figure \ref{fig:Scaling with dt} shows how the Kronecker preconditioners perform as $\Delta t$ increases, and the difference between linearised model operators increases. We compare against the approach of \cite{tabeart2021saddle} using $k=N+1$, i.e. $\widehat{\mathbf{L}} \equiv \mathbf{L}$. For the SPD problem, we also  plot convergence for the vectorised approach with LMP. For $\Delta t \in [10^{-6},10^{-2}]$ both {\sc matCG/matGMRES} and {\sc vecCG/vecGMRES} behave similarly, with only a small increase in iterations with $\Delta t$. However, for larger values of $\Delta t$ the number of iterations required to reach convergence increases for all methods, with the largest impact seen for {\sc matCG/matGMRES}. The use of $r=p$ within {\sc matCG/matGMRES} is much more sensitive to the choice of $\Delta t$, with a steady increase in the number of iterations required to reach convergence as $\Delta t$ increases. For $\Delta t=5\times 10^{-2}$ there is negligible benefit in terms of iterations to the inclusion of observation information within the Schur complement. For large values of $\Delta t$, the difference between $\widehat{\mathbf{L}} = I_{N+1}\otimes I_s-\Sigma\otimes \widehat{M}$ and $\mathbf{L}$ increases, meaning that including $r>0$ factors coming from the observation term is obtaining an improved estimate of the wrong preconditioner. In the future it might be possible to design alternative additional terms that can correct for this discrepancy.  We note that for LMP as $\widehat{\mathbf{L}}\equiv \mathbf{L}$ the `correct' low-rank update is used for all choices of $\Delta t$.    }

						\subsection{Results for heat equation}\label{Results for heat equation}
						We now present an example with  $M_i=M$ for all $i$. This simpler setting allows us to validate the theoretical properties of our new approach, and consider computational aspects such as the cost of the approach proposed in section~\ref{More performing Schur-complement approximations} and scaling of our methods with increasing observation times. 
						Our second numerical example comes from~\cite[][Section~6.1]{tabeart2021saddle}. For this example, $\mathbf{L}$ is a Stein operator of the form~\eqref{eq:Ltimeindep} where the matrix $M$ amounts to the discrete operator stemming from the discretization of the one-dimensional heat equation on the unit line 
						\begin{equation*}
							\frac{\partial u}{\partial t} = \frac{\partial^2 u}{\partial x^2}
						\end{equation*}
						with Dirichlent boundary conditions $u(0,t) = u_0$, $u(1,t) = u_1$ for all $t\in(0,1]$. 
						By discretising the equation above by means of the forward Euler method in time and second-order central differences in space, $M$ can be written as follows
						$$M=\begin{bmatrix}
							0 & 0 & 0 & 0 &\cdots & 0 & 0\\
							0 & 1-2r & r & 0 &\cdots & 0 & 0\\
							0 & r & \ddots & \ddots &  & \vdots & \vdots \\
							0 & 0 & \ddots & \ddots & \ddots & 0& 0 \\
							\vdots &  & \ddots& \ddots& \ddots& r & 0\\
							0 &\cdots & 0& 0& r& 1-2r& 0\\
							0 &\cdots & 0& 0& 0& 0 & 0\\
						\end{bmatrix},\quad r=\frac{\Delta t}{(\Delta x)^2}.
						$$
						For this example we use $\Delta x = 10^{-3}$ and $\Delta t = 4\times 10^{-7}$.
						
						We begin by comparing the performance of the matrix and vector oriented approaches  when using the exact $\widehat{\mathbf{L}}\equiv\mathbf{L}$ in all four choices of preconditioner proposed in this paper, and $r=0$ in~\eqref{eq:Woodbury expression}. For {\sc vecCG/vecGMRES}, $\mathbf{L}^{-1}$ is computed by using the procedure coming from~\cite{tabeart2021saddle} for $k=N+1$. We remind the reader that the two versions of CG/GMRES are equivalent in exact arithmetic. However, the matrix oriented approaches present some computational advantages for this example.  Table \ref{tab:Mat vs vec GMRES exact L} demonstrates the remarkable gain in efficiency that occurs for $\widehat{\mathbf{S}}$, $\mathcal{P}_{  \mathcal{D}}$, and $\mathcal{P}_{\mathcal{T}}$ for $N=10$.  We notice that the strategy based on 
						$\mathcal{P}_{\mathcal{C}}$ is  faster than the ones related to the block diagonal and block triangular preconditioners, in spite of the fact that it requires a larger number of iterations. This is due to the small $N$ selected in this example which makes the inversion of 
						$\mathbf{D}$ the most expensive step in the preconditioners that involve $\mathbf{\widehat{S}}$. We recall that $\mathcal{P}_{\mathcal{C}}$ avoids the application of $\mathbf{D}^{-1}$. Moreover, for this choice of $N$, {\sc vecGMRES} is  faster than {\sc matGMRES} when
						$\mathcal{P}_{\mathcal{C}}$ is adopted as preconditioning operator. This is related to the cost of the eigendecomposition of $M$ described at the end of section~\ref{On the inversion of the Stein operator}. This step cubically depends on $s$; taking about 0.2s for this problem, namely about $1/2$ of the overall running time achieved by {\sc matGMRES}. Nevertheless, we anticipate the cost of this eigendecomposition to be amortised for larger choices of $N$ (see e.g. Figure \ref{fig: Iterations large heat eqn}).
						
						\begin{table}[t]
							\centering
							\begin{tabular}{r|c |c c c}
								& $\widehat{\mathbf{S}}$ & $\mathcal{P}_{\mathcal{D}}$ & $\mathcal{P}_{\mathcal{T}}$ & $\mathcal{P}_{\mathcal{C}}$ \\
								\hline
								Iterations -- {\sc vecCG/vecGMRES} &18.4  & 36.6 & 18.8 & 37.6\\
								Iterations -- {\sc matCG/matGMRES} &18.4  & 36.6 & 18.8 & 37.6\\
								Wallclock time -- {\sc vecCG/vecGMRES} & 6.2475 &  7.7263 &   4.1889  &  0.4013\\ 
								Wallclock time -- {\sc matCG/matGMRES} & 1.3356 &0.8472  & 0.8493  &  0.5713
							\end{tabular}
							\caption{Example~\ref{Results for heat equation}. Iterations and wallclock time to convergence for
								$N = 10$ and $\widehat{\mathbf{L}} \equiv \mathbf{L}$ for the objective function formulation~\eqref{eq:CGproblem} (column 1) and the saddle point formulation~\eqref{eq:saddle-point} with different preconditioners (columns 2-4) averaged over 10 realisations.}
							\label{tab:Mat vs vec GMRES exact L}
						\end{table}
						
						
						We now focus on the performance achieved by the novel Schur complement approximation~\eqref{eq:Shat_exp} and its implementation illustrated in section~\ref{Algorithmic considerations}. To this end, we consider only $\widehat{\mathbf{S}}$ and $\mathcal{P}_{\mathcal{D}}$. Similar results have also been obtained for $\mathcal{P}_{\mathcal{T}}$ but are not presented here. In Table~\ref{tab:my_label2.1} (left), for different values of $N$, we report the iteration count and the overall running time of {\sc matCG/matGMRES} when the approximate Schur complement $\widehat{\mathbf{S}}$ is adopted for $\mathbf{S}$ and $\mathcal{P}_{\mathcal{D}}$ for $r=0$ and $r=p$.
						We  notice that choosing $r=p$ in~\eqref{eq:Shat_exp} leads to a remarkable decrease in the number of iterations needed to converge. Moreover, the number of CG/GMRES iterations we perform turns out to be $N$-independent for both problem. 
						
						We notice that the number of iterations required to solve the {\sc matCG} formulation is smaller than for {\sc matGMRES}, leading to smaller wallclock times. We note that for $r=p$ we are approximating the inverse of $\mathbf{S}$ to a small tolerance, and hence obtain convergence in a single iteration of {\sc matCG}. Wallclock times for the case $r=p$ are comparable for both problems. {\color{black}
							Even though the use of inner CG introduces some inexactness in the preconditioning step, we notice that the number of iterations performed by {\sc matCG/matGMRES} with $\widehat{\mathbf{S}}$ and $r=p$ is independent of $N$ and equal to the number of iterations expected in case of an exact
							computation of $\widehat{\mathbf{S}}^{-1}$.
						}
						
						The large reduction in the iteration count also leads to a significant speed-up of the overall solution process, which is not obvious in general. Indeed, the use of~\eqref{eq:Woodbury expression} for large values of $r$ can be computationally demanding due to the need to solve the linear system with $I_{r(N+1)}+ \mathbf{K}_r^T\mathbf{L}^{-1}\mathbf{D}\mathbf{L}^{-T}\mathbf{K}_r$. However, thanks to the matrix CG strategy presented in section~\ref{Algorithmic considerations}, which takes full advantage of the semi-diagonalization of $\mathbf{L}$, dealing with $\widehat{\mathbf{S}}$ by~\eqref{eq:Woodbury expression} turns out to be computationally affordable for $r=p$. In Table~\ref{tab:my_label2.1} (far right), we report the number of CG iterations and the related running time needed to approximately invert $\widehat{\mathbf{S}}$ within the {\sc matGMRES} iteration. We remind the reader that, in light of Theorem~\ref{th:P_Dbasisvectors}, $\widehat{\mathbf{S}}^{-1}$ has to be computed only every other GMRES iteration. From the results in Table~\ref{tab:my_label2.1} (right), we can see that the CG steps
						correspond to a small proportion of the overall GMRES running time for small $N$. However, as $N$ increases, the cost of the inner CG iteration becomes a larger proportion of the overall wallclock time. The number of CG iterations increases with $N$, leading to a more demanding preconditioning step for larger numbers of observation times.
						In this scenario, equipping the inner CG solve with effective preconditioning operators may be largely beneficial. However, as we previously mentioned, natural preconditioning candidates were not able to reduce the CG iteration count without significantly increasing its computational cost per iteration. We will explore this challenging topic of designing bespoke preconditioners for the inner CG solver in the future. 
						

						\begin{table}[t]
							\centering
							\begin{tabular}{c|c c || c c |c c }
								& \multicolumn{2}{c||}{{\sc matCG}} &\multicolumn{2}{c|}{{\sc matGMRES}} & \multicolumn{2}{c}{Inner CG}  \\
								&&&&&\multicolumn{2}{c}{2nd GMRES it.} \\ 
								$N$& Its. &  Time &  Its. & Time & Its. & Time   \\
								\hline
								10 ($r=0$) &18.3& 1.1732 & 36.4& 1.4561 && \\
								~~~  ($r=p$) &1&0.6683 &3 & 0.6332 & 20 & 0.0598\\
								\hline
								20  ($r=0$) &26.1& 1.8153 & 51.8 & 2.1692&&\\
								~~~  ($r=p$)&1& 0.8008 &3 &0.7577 & 30.3 & 0.1420\\
								\hline
								30  ($r=0$) &34.2& 2.3684 &63.0 & 2.9455&& \\
								~~~  ($r=p$) &1& 0.9045 &3& 0.9111 &40 & 0.2674 \\
								\hline
								40  ($r=0$)& 42.1 & 2.9996& 74.2 & 4.0273&& \\
								~~~  ($r=p$)&1&1.0495 &3 & 1.0598 &49.3 &  0.4227 \\
								\hline 
								50 ($r=0$)&49.5& 3.5876 &83.0 & 5.1370&&\\
								~~~  ($r=p$)&1& 1.2258 &3 & 1.2289 & 59 & 0.5731\\
								\hline
								60  ($r=0$)&57.1&4.2990 & 92.6 & 6.7560&&\\
								~~~  ($r=p$) & 1& 1.5402 & 3&  1.5267 &67 & 0.8732
							\end{tabular}
							\caption{Example~\ref{Results for heat equation}. Iterations and wallclock time to convergence for  $\widehat{\mathbf{S}}$ (first two columns)
								$\mathcal{P}_{\mathcal{D}}$ (columns 3-6) with the two different choices $r=0$ and $r=p$ in~\eqref{eq:Shat_exp} as $N$ varies. In columns 5 and 6 we report the iterations and wallclock time needed by inner CG method to solve the linear system with $I_{r(N+1)}+ \mathbf{K}_r^T\mathbf{L}^{-1}\mathbf{D}\mathbf{L}^{-T}\mathbf{K}_r$ involved in~\eqref{eq:Woodbury expression} when $r>0$ in $\mathcal{P}_{\mathcal{D}}$. }
							\label{tab:my_label2.1}
						\end{table}
						

					We conclude the heat equation example by comparing the novel preconditioning operators developed in this paper with  state-of-the-art approaches. In particular, we consider {\sc matCG} equipped with $\widehat{\mathbf{S}}$ and {\sc matGMRES} equipped with $\mathcal{P}_\mathcal{D}$, $\mathcal{P}_\mathcal{T}$ (with both $r=0$ and $r=p$ in~\eqref{eq:Shat_exp}), and $\mathcal{P}_\mathcal{C}$.  Using $k=N+1$ within the strategy of~\cite{tabeart2021saddle} becomes infeasible for large values of $N$. We therefore use the strategy coming from~\cite{tabeart2021saddle} with $k=3$ to approximate $\widehat{\mathbf{L}}^{-1}$ as a comparison. This is implemented with  {\sc vecCG/vecGMRES} and is otherwise equipped with the same preconditioning frameworks as our novel approach.

					\begin{figure}[t]
						\centering
						\includegraphics[width=0.9\textwidth,trim=0mm 0mm 0mm 0mm, clip]{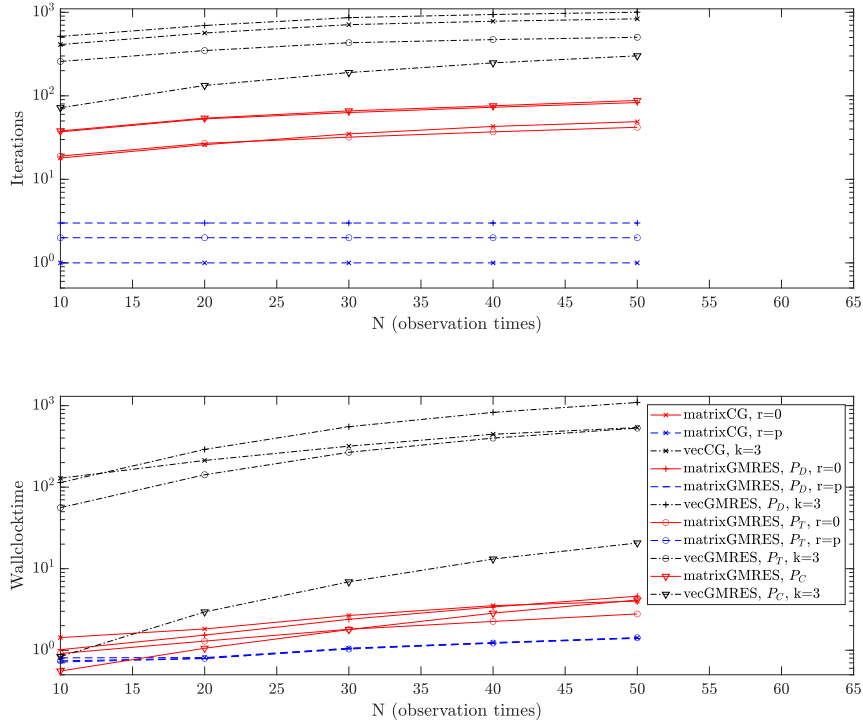}
						\caption{Example~\ref{Results for heat equation}. Iterations (top) and wallclock time (bottom) to reach convergence with increasing problem size (number of observation times) for the heat equation problem for different choices of preconditioner. Red solid lines denote {\sc matCG/matGMRES} with $r=0$, blue dashed lines denote {\sc matCG/matGMRES} with $r=p$, and black dot-dashed lines denote {\sc vecCG/vecGMRES} with $k=3$. Crosses denote $\widehat{\mathbf{S}}$, pluses denote $\mathcal{P}_\mathcal{D}$, circles denote $\mathcal{P}_\mathcal{T}$ and triangle denote $\mathcal{P}_\mathcal{C}$. 
							We report averaged behaviour over 10 realisations. }\label{fig: Iterations large heat eqn}
					\end{figure}
					
					Figure~\ref{fig: Iterations large heat eqn} (top) shows the number of iterations to reach convergence with an increasing number of observation times $N$.
					We can observe that $\widehat{\mathbf{S}}$, $\mathcal{P}_\mathcal{D}$ and $\mathcal{P}_\mathcal{T}$ require a very small number of iterations whenever we select $r=p$ in~\eqref{eq:Shat_exp}. The performance of these operators is optimal as the number of performed CG/GMRES iterations is constant with increasing $N$. We recall that such optimality is not guaranteed, as the  linear system 
					with $I_{r(N+1)}+ \mathbf{K}_r^T\mathbf{L}^{-1}\mathbf{D}\mathbf{L}^{-T}\mathbf{K}_r$ involved in~\eqref{eq:Woodbury expression} is solved iteratively to a relative residual tolerance of  $10^{-8}$.  As previously mentioned, such \emph{inexactness} does not allow us to claim that we are working within the scenarios depicted in Corollary~\ref{cor:diag_prec}--\ref{cor:triang_prec} -- for which we would be able to guarantee an $N$-independent number of CG/GMRES iterations -- even though $r=p$ in~\eqref{eq:Shat_exp}. Nevertheless, for this example the overall solution process does demonstrate $N$ independence of iterations to reach convergence.
					
					Competitive performance is attained also when $r=0$ in the 
					$\widehat{\mathbf{S}}$, $\mathcal{P}_\mathcal{D}$, and $\mathcal{P}_\mathcal{T}$
					preconditioning frameworks compared to the $\mathbf{L}^{-1}$ approximation with $k=3$. The large difference in iterations for the strategy coming from~\cite{tabeart2021saddle} also leads to much longer wallclock times.
					
					We notice that the use of $\mathcal{P}_\mathcal{C}$ leads to a number of GMRES iterations which is always very similar to the one achieved by $\mathcal{P}_\mathcal{D}$ with $r=0$. The performance of $\widehat{\mathbf{S}}$ is also similar to the performance of $\mathcal{P}_\mathcal{T}$ in the $r=0$ setting. 
					

					In Figure~\ref{fig: Iterations large heat eqn} (bottom) we report the computational time of the overall CG/GMRES solution process for all the preconditioning operators we mentioned above. We can see that, except for $\mathcal{P}_\mathcal{C}$ with small $N$, {\sc vecGMRES} is much slower than {\sc matGMRES} equipped with our novel preconditioning strategies. In particular, from the results depicted in Figure~\ref{fig: Iterations large heat eqn} (bottom) we see that selecting $r=p$ in $\widehat{\mathbf{S}}$, $\mathcal{P}_\mathcal{D}$ and $\mathcal{P}_\mathcal{T}$ is a favourable choice over $r=0$ also in terms of running time, with competitive scaling with $N$ when using the improved Schur complement approximation. 

					\section{Conclusions and outlook}\label{Conclusions and outlook}
					{\color{black}
						To fully exploit the rich Kronecker structure of the matrices stemming from weak-constraint 4D-Var problems, matrix-oriented Krylov methods can be employed to solve both~\eqref{eq:CGproblem} and~\eqref{eq:saddle-point}. The use of such machinery naturally leads to the design of new preconditioning approaches. In particular, by selecting a fresh option for the operator $\widehat{\mathbf{L}}$ whose inversion can be recast in terms of the solution of a Stein matrix equation, we designed improved preconditioners able to drastically reduce the Krylov iteration count for certain problems. Our new approach also allows for the efficient inclusion of information from the observation term of the Schur complement $\mathbf{S}$, leading to more accurate approximations $\widehat{\mathbf{S}}$. 
						
						In the case of {\color{black}observation-time} independent forecast models $\mathcal{M}$, our new preconditioning frameworks achieve  optimal performance in terms of the number of iterations,  remarkably without increasing the computational cost of the overall solution process. 
					}
					
					{\color{black} The implementation presented in this paper requires a number of assumptions on the structure of the data assimilation system, which we hope to relax in future work. Firstly, }
					the machinery developed here relies on having a moderate spatial dimension $s$. This assumption is crucial for the Stein equation solution scheme presented in section~\ref{On the inversion of the Stein operator}. We plan to extend the preconditioning framework presented in this paper to the case of sizable $s$ in near future. This can be achieved, e.g., by using projection-based methods for large-scale Stein equations.
					{\color{black}The approach currently also requires that a number of other components of the assimilation problem have a strict Kronecker structure, meaning that the model error and the observing system are constant for all observation times.} 
				As reported in section~\ref{The case of different Ms}, we could approximate each term at the preconditioning level by means of some Kronecker forms. However, the selection of such approximations may be cumbersome. These aspects will be investigated elsewhere.
				
				\section*{Acknowledgements}
				We thank Adam El-Said for his code for the Lorenz 96 weak constraint 4D-Var assimilation problem. We also thank Ieva Dau{\v{z}}ickait{\.e} for providing us with her code for the LMP implementation of \cite{dauvzickaite2021randomised}.
				
				The first author is member of the Italian INdAM Research group GNCS. His work was partially supported by the research project ``Tecniche avanzate per problemi evolutivi: discretizzazione, algebra lineare numerica, ottimizzazione'' (INdAM - GNCS  Project CUP\_E55F22000270001).
				
				The second author gratefully acknowledges funding from the Engineering and Physical Sciences Research Council (EPSRC) grant EP/S027785/1.
				
				\section*{Appendix A}\label{Appendix}
				Here we report the proof of Theorem~\ref{th:P_Dbasisvectors}.
				
				\begin{proof}
					We first write 
					$$\mathcal{AP}_{\mathcal{D}}^{-1}=
					\begin{bmatrix}
						I & 0 & \mathbf{L}\widehat{\mathbf{S}}^{-1}\\
						0 & I & \mathbf{H}\widehat{\mathbf{S}}^{-1}\\
						\mathbf{L}^T\mathbf{D}^{-1} & \mathbf{H}^T\mathbf{R}^{-1} & 0\\
					\end{bmatrix}.
					$$
					We show the statement by induction on $k\geq 1$.
					
					For $k=1$, we define
					$$v_1=\frac{1}{\sqrt{\|b\|^2+\|d\|^2}}\begin{bmatrix}
						b\\
						d\\
						0
					\end{bmatrix}.
					$$
					Then, 
					$$\widetilde v_2=\mathcal{AP}_{\mathcal{D}}^{-1}v_1=
					\frac{1}{\sqrt{\|b\|^2+\|d\|^2}}\begin{bmatrix}
						b\\
						d\\
						\mathbf{L}^T\mathbf{D}^{-1}b + \mathbf{H}^T\mathbf{R}^{-1}d
					\end{bmatrix},
					$$
					and the latter vector needs to be orthogonalized with respect to $v_1$. 
					A direct computation shows that the outcome of this orthogonalization is 
					$$ \widehat v_2=
					\frac{1}{\sqrt{\|b\|^2+\|d\|^2}}\begin{bmatrix}
						0\\
						0\\
						\mathbf{L}^T\mathbf{D}^{-1}b + \mathbf{H}^T\mathbf{R}^{-1}d
					\end{bmatrix},
					$$
					and $v_2=\widehat v_2/\|\widehat v_2\|$.
					
					We now assume that the result has been shown for a certain $\bar k>1$ and we prove the inductive step for $\bar k+1$.
					
					It holds
					$$\widetilde v_{2(\bar k+1)-1}=\widetilde v_{2\bar k+1}=\mathcal{AP}_{\mathcal{D}}^{-1}v_{2\bar k}=\begin{bmatrix}
						\mathbf{L}\widehat{\mathbf{S}}^{-1}z_{2\bar k}\\
						\mathbf{H}\widehat{\mathbf{S}}^{-1}z_{2\bar k}\\
						0
					\end{bmatrix}.
					$$
					Then the orthogonalization step is such that 
					$$\widehat v_{2(\bar k+1)-1}
					=\widetilde v_{2(\bar k+1)-1}-\sum_{j=1}^{\bar k}\alpha_j
					v_{2j-1}-\sum_{j=1}^{\bar k}\beta_j
					v_{2j},
					$$
					and the only term that may potentially contribute to the third block of $\widehat v_{2(\bar k+1)-1}$ is 
					$$\sum_{j=1}^{\bar k}\beta_j
					v_{2j}=\begin{bmatrix}
						0\\
						0\\
						\sum_{j=1}^{\bar k}\beta_j
						z_{2j}
					\end{bmatrix}.$$ 
					However, all the scalars $\beta_j$'s are zero since
					$$
					\beta_j=\widetilde v_{2(\bar k+1)-1}^Tv_{2j}=
					\begin{bmatrix}
						(\mathbf{L}\widehat{\mathbf{S}}^{-1}z_{2\bar k})^T,
						( \mathbf{H}\widehat{\mathbf{S}}^{-1}z_{2\bar k})^T,
						0
					\end{bmatrix}\begin{bmatrix}
						0\\
						0\\
						z_{2j}
					\end{bmatrix}=0.
					$$
					Therefore, $v_{2(\bar k+1)-1}=\widehat v_{2(\bar k+1)-1}/\|\widehat v_{2(\bar k+1)-1}\|$ has a third zero block.
					
					To conclude, if $v_{2(\bar k+1)-1}=[u_{2(\bar k+1)-1}^T,w_{2(\bar k+1)-1}^T,0]^T$, then
					$$\widetilde v_{2(\bar k+1)}=\mathcal{AP}_{\mathcal{D}}^{-1}v_{2(\bar k+1)-1}=\begin{bmatrix}
						u_{2(\bar k+1)-1}\\
						w_{2(\bar k+1)-1}\\
						\mathbf{L}^T\mathbf{D}^{-1}u_{2(\bar k+1)-1}+
						\mathbf{H}^T\mathbf{R}^{-1}w_{2(\bar k+1)-1}
					\end{bmatrix},
					$$
					and orthonormalizing such a vector with respect to the computed basis, and in particular $v_{2(\bar k+1)-1}$, leads to a 
					$v_{2(\bar k+1)}$ whose third block is the only nonzero block. Hence the result of Theorem 2.1 holds by induction.
				\end{proof}
				
				Here we report the proof of Proposition~\ref{prop:Mi not equal M}.
				
				{\color{black}
					\begin{proof}
						
						We begin by observing that $\widehat{\mathbf{L}}^{-T}\mathbf{L}^T\mathbf{L}\widehat{\mathbf{L}}^{-1} = I + A(\widehat{M})$ where 
						the $(i,j)$th block of $A(\widehat{M})$ is given by
						\begin{equation}
							\begin{cases}
								\sum_{k=i}^N \widehat{M}^{(k-i)T}D_k^TD_k\widehat{M}^{k-i} & \emph{ if } i=j,\\
								D_{i-1}\widehat{M}^{i-j-1}+ \sum_{k=i}^N\widehat{M}^{(k-i)T}D_k^TD_k\widehat{M}^{k-j} & \emph{ if } i>j, \\
								\widehat{M}^{(j-i-1)T}D_{j-1}^T+ \sum_{k=j}^N\widehat{M}^{(k-i)T}D_k^TD_k\widehat{M}^{k-j} & \emph{ if } j>i.\\
							\end{cases}
						\end{equation}
						We then write $A(\widehat{M}) = \sum_{m=1}^N A_m$ where $A_m$ contains terms which depend only on $D_m$ and powers of $\widehat{M}$ and their transposes only.   We bound the eigenvalues of $A(\widehat{M})$ above by applying \cite[][Equation 5.12.2]{bernstein} to obtain
						\begin{equation}
							\lambda_{\max}(A(\widehat{M}))\le 1+ \sum_{k=1}^N \lambda_{\max}({A_k}).
						\end{equation}
						
						We now bound the eigenvalues of $A_m$.  For $m=1,\dots, N$, the $(i,j)$th block of $A_m$ is given by
						
						\begin{equation}
							\begin{cases}
								\widehat{M}^{(m-i)T}D_m^TD_m\widehat{M}^{m-j} & \emph{ if } i,j\le m,\\
								\widehat{M}^{(m-i)T}D_m^T & \emph{ if } j=m+1>i, \\
								D_m\widehat{M}^{m-i} & \emph{ if } i=m+1>j.
							\end{cases}
						\end{equation}
						For all choices of $m$, $A_m$ has $m^2-1$ non-zero blocks, and has rank $2s$. If $0_\ell\in\mathbb{R}^{\ell}$ denotes the zero vector of length $\ell$, the $(m-2)s$ eigenvectors corresponding to zero take the form $(e_i,-\widehat{M}e_i,0_{(m-2)s})$, or $(0_{s},e_i,-\widehat{M}e_i,0_{(m-1)s}, \dots, (0_{(m-1)s},e_i,-\widehat{M}e_i,0_{s})$.
						
						The non-zero eigenvalues of $A_m$ can be found by solving the $s\times s$ system
						\begin{equation*}
							\left(D_m\left(\sum_{k=0}^{N-1}\widehat{M}^k\widehat{M}^{kT}\right)D_m^T\right) v = \frac{\mu^2}{\mu+1} v,
						\end{equation*}
						i.e. 
						\begin{equation}\label{eq:mu}
							\mu = 0.5(\rho\pm\sqrt{\rho^2+4\rho}),
						\end{equation}
						where $\rho$ are the eigenvalues of $(D_m(\sum_{k=0}^{m-1}\widehat{M}^k\widehat{M}^{kT})D_m^T) $.
						
						By the monotonicity of \eqref{eq:mu}, the largest value of $\mu$ occurs for the largest value of $\rho$ with the positive option, and the smallest value of $\mu$ occurs for the largest value of $\rho$ taking the negative option. Therefore an upper bound for $\rho$ provides us with an upper bound for $\mu$, and hence $\lambda_{\max}(A(\widehat{M}))$.
						
						By similarity 
						\begin{align*}
							\max(\rho_m) &= \lambda_{\max}(D_m^T D_m(\sum_{k=0}^{m-1}\widehat{M}^k\widehat{M}^{kT}))\\
							& \le \lambda_{\max}(D_m^T D_m)\lambda_{\max} (\sum_{k=0}^{m-1}\widehat{M}^k\widehat{M}^{kT}))\\ 
							& \le \lambda_{\max}(D_m^T D_m) \sum_{k=0}^{m-1}\lambda_{\max}(\widehat{M}^T\widehat{M})^k.
						\end{align*}
						
						A loose upper bound can be obtained by defining \begin{equation*} \rho_N = \max_{m}\lambda_{\max}(D_m^TD_m)\sum_{k=0}^{N-1}\lambda_{\max}(\widehat{M}^T\widehat{M})^k.
						\end{equation*}
						Moreover,
						$$
						\sum_{k=0}^{N-1}\lambda_{\max}(\widehat{M}^T\widehat{M})^k= \left\{\begin{array}{ll}
							N, &  \text{if}\;\lambda_{\max}(\widehat{M}^T\widehat{M})=1,\\
							&\\
							\frac{1-\lambda_{\max}^N(\widehat{M}^T\widehat{M})}{1-\lambda_{\max}(\widehat{M}^T\widehat{M})},
							& \text{otherwise}.
						\end{array}\right.
						$$

						For every choice of $m$ it holds $\mu_m \le 0.5(\rho_N+\sqrt{\rho_N^2+4\rho_N})$, therefore 
						\begin{equation*}
							\lambda_{max}(A(\widehat{M})) \le 1 + \frac{N}{2}(\rho_N+\sqrt{\rho_N^2+4\rho_N}).
						\end{equation*}
					\end{proof}
					
					Here we report the proof of Proposition~\ref{Prop:SteinSol}.
					
					\begin{proof}
						We show~\eqref{eq:SolStein}. The proof for~\eqref{eq:SolTStein} is analagous.
						
						By plugging~\eqref{eq:Sigmaform} into~\eqref{eq:Stein_eq} we get 
						$$Z-\widehat{M}ZC^T+\widehat{M}Ze_{N+1}e_1^T=V.$$
						Premultiplying by $T^{-1}$ and postmultiplying by $F^T$ yields
						$$\widetilde Z-\Lambda\widetilde Z\Pi+\Lambda\widetilde Z(F^{-T}e_{N+1})(e_1^TF^T)=T^{-1} VF^T,\qquad \widetilde Z:=T^{-1}ZF^T, $$
						whose Kronecker form is given by
						$$(I_{N+1}\otimes I_s-\Pi\otimes\Lambda+(Fe_1)(e_{N+1}^TF^{-1})\otimes \Lambda)\text{vec}(\widetilde Z)=\text{vec}(T^{-1} VF^T).$$
						If $\mathbf{G}:=I_{N+1}\otimes I_s-\Pi\otimes\Lambda\in\mathbb{C}^{(N+1)s\times(N+1)s}$,
						$\mathbf{M}:=Fe_1\otimes\Lambda\in\mathbb{C}^{(N+1)s\times s}$, and $\mathbf{N}:=F^{-T}e_{N+1}\otimes I_s\in\mathbb{C}^{(N+1)s\times s}$, the Sherman-Morrison-Woodbury formula~\cite[][Equation (2.1.4)]{GolubVanLoan13} shows that 
						$$\text{vec}(\widetilde Z)=\mathbf{G}^{-1}\text{vec}(T^{-1} VF^T)-\mathbf{G}^{-1}\mathbf{M}(I+\mathbf{N}^T\mathbf{G}^{-1}\mathbf{M})^{-1}\mathbf{N}^T\mathbf{G}^{-1}\text{vec}(T^{-1} VF^T).
						$$
						Once $\widetilde Z$ is computed, we retrieve $Z$ by performing $Z=T\widetilde Z F^{-T}$.
						
						We now derive a cheap procedure for the computation of $\widetilde Z$ which does not involve the construction of any large matrix.
						
						We first notice that, since $\mathbf{G}$ is diagonal it holds
						$$\text{vec}(Y)=\mathbf{G}^{-1}\text{vec}(T^{-1}VF^T) \quad\Longleftrightarrow\quad Y=P\circ (T^{-1}VF^T),$$
						where $P\in\mathbb{C}^{s\times(N+1)}$ is such that $P_{i,j}=1/(1-\lambda_i\pi_j)$.
						
						Moreover, by exploiting the Kronecker structure of $\mathbf{N}$, we have 
						$$\mathbf{N}^T\mathbf{G}^{-1}\text{vec}(T^{-1} VF^T)=((e_{N+1}^TF^{-1})\otimes I_s)\text{vec}(Y)=YF^{-T}e_{N+1}.$$

						We now focus on the computation of $\mathbf{N}^T\mathbf{G}^{-1}\mathbf{M}$. We remind the reader that 
						$$\mathbf{G}=\begin{pmatrix}
							I_s-\pi_1\Lambda & & \\
							& \ddots & \\
							& & I_s-\pi_{N+1}\Lambda\\
						\end{pmatrix},\; \mathbf{N}=\begin{pmatrix}
							(F^{-T}e_{N+1})_1I_s\\
							\vdots\\
							(F^{-T}e_{N+1})_{N+1}I_s\\
						\end{pmatrix},\,
						\mathbf{M}=\begin{pmatrix}
							(Fe_{1})_1\Lambda\\
							\vdots\\
							(Fe_{1})_{N+1}\Lambda\\
						\end{pmatrix},
						$$
						and a direct computation shows that 
						$$\mathbf{N}^T\mathbf{G}^{-1}\mathbf{M}=\sum_{j=1}^{N+1}(I_s-\pi_j\Lambda)^{-1}\Lambda(Fe_1)_j(F^{-T}e_{N+1})_j=\text{diag}(P(\Lambda Fe_1\circ F^{-T}e_{N+1})).$$
						The formulation above provides a cheap expression for the construction of $\mathbf{N}^T\mathbf{G}^{-1}\mathbf{M}$ along with illustrating its diagonal structure, hence solving the linear system with ${\color{black}U}=I+\text{diag}(P(\Lambda Fe_1\circ F^{-T}e_{N+1}))$ does not significantly increase the cost of computing $Z$.
						
						Returning to the computation of $\widetilde Z$, we have 
						\begin{align*}
							\mathbf{G}^{-1}\mathbf{M}(I+\mathbf{N}^T\mathbf{G}^{-1}\mathbf{M})^{-1}\mathbf{N}^T\mathbf{G}^{-1}\text{vec}(T^{-1} VF^T)=& \mathbf{G}^{-1}\mathbf{M}{\color{black}U}^{-1}YF^{-T}e_{N+1}\\
							=&P\circ(\Lambda {\color{black}U}^{-1}YF^{-T}e_{N+1}e_1^TF^{T})\\
							=&W.
						\end{align*}
						Combining the steps above yields the statement in~\eqref{eq:SolStein}.
					\end{proof}

				}
				
				\section*{Appendix B}\label{AppendixB}
				For the sake of completeness, in Algorithm~\ref{matGMRES} we report the pseudocode of the matrix-oriented GMRES method applied to~\eqref{eq:saddle-point}. The $m$-th basis vector of the Krylov subspace $K_m(\mathcal{A},\mathbf{b})$ is represented in terms of the matrices $\mathcal{V}_{1,m}\in\mathbb{R}^{s\times (N+1)}$, $\mathcal{V}_{2,m}\in\mathbb{R}^{p\times (N+1)}$, and $\mathcal{V}_{3,m}\in\mathbb{R}^{s\times (N+1)}$, namely 
				$$v_m=\text{vec}\begin{bmatrix}
					\mathcal{V}_{1,m}\\
					\mathcal{V}_{2,m}\\
					\mathcal{V}_{3,m}
				\end{bmatrix}.
				$$
				The residual norm in line~\ref{alg:rescomp} can be cheaply computed by following, e.g., the classic Given rotation approach presented in~\cite[Section 6.5.3]{Saad03}.
				
				{\color{black}
					Similarly, in Algorithm~\ref{matCG} we report the pseudocode of the matrix-oriented CG method applied to~\eqref{eq:CGproblem}.
					
					In what follows, $(A)_{i,j}$ will denote the $(i,j)$-th entry of the matrix $A$.
				}
				\setcounter{AlgoLine}{0}
				\begin{algorithm}[t]
					\SetKwInOut{Input}{input}\SetKwInOut{Output}{output}
					\Input{$B,Q_1,\ldots,Q_N\in\mathbb{R}^{s\times s},$ $R_0,\ldots,R_N\in\mathbb{R}^{p\times p}$, $H_0,\ldots,H_N\in\mathbb{R}^{p\times s}$, $M_1,\ldots,M_N\in\mathbb{R}^{s\times s}$ $[b_0,c_1,\ldots,c_N]\in\mathbb{R}^{s\times (N+1)}$, $[d_0,\ldots,d_N]\in\mathbb{R}^{p\times (N+1)}$, $m_{\max}$, $\varepsilon>0$}
					\Output{$\delta\Theta_m,\delta X_m\in\mathbb{R}^{s\times (N+1)}$, $\delta\Lambda_m^T\in\mathbb{R}^{p\times (N+1)}$ such that $\text{vec}(\delta\Theta_m^T,\delta\Lambda_m^T,\delta X_m^T)^T$ is an approximate solution to~\eqref{eq:saddle-point}.}
					\BlankLine
					Compute $\beta=\sqrt{\|[b_0,c_1,\ldots,c_N]\|_F^2+\|[d_0,\ldots,d_N]\|_F^2}$ and set $\mathcal{V}_{1,1}= [b_0,c_1,\ldots,c_N]/\beta$,
					$\mathcal{V}_{2,1}= [d_0,\ldots,d_N]/\beta$, and $\mathcal{V}_{3,1}=0$ \;
					
					\For{$m=1, 2,\dots,$ till $m_{\max}$}{
						
						Set 
						{\small $$\widehat{\mathcal{V}}_{1,m+1}=[B\mathcal{V}_{1,m}e_1,Q_1\mathcal{V}_{1,m}e_2,\ldots,Q_N\mathcal{V}_{1,m}e_{N+1}]+[\mathcal{V}_{3,m}e_1,\mathcal{V}_{3,m}e_2-M_1\mathcal{V}_{3,m}e_1,\ldots,\mathcal{V}_{3,m}e_{N+1}-M_N\mathcal{V}_{3,m}e_{N}]$$
							$$\widehat{\mathcal{V}}_{2,m+1}=[R_0\mathcal{V}_{2,m}e_1,\ldots,R_N\mathcal{V}_{2,m}e_{N+1}]+[H_0\mathcal{V}_{3,m}e_1,\ldots,H_N\mathcal{V}_{3,m}e_{N+1}]$$
							$$\widehat{\mathcal{V}}_{3,m+1}=[\mathcal{V}_{1,m}e_1-M_1^T\mathcal{V}_{1,m}e_2,\ldots,\mathcal{V}_{1,m}e_N-M_N^T\mathcal{V}_{1,m}e_{N+1},\mathcal{V}_{1,m}e_{N}]
							+[H_0^T\mathcal{V}_{2,m}e_1,\ldots,H_N^T\mathcal{V}_{2,m}e_{N+1}]$$
						}\;

						Set $(\underline T_m)_{j,m}=0$ for $j=1,\ldots,m+1$\;
						
						\For{$\ell=1,2$}{
							Compute $$(\underline T_m)_{j,m}=(\underline T_m)_{j,m}+\sqrt{\text{trace}(\widehat{\mathcal{V}}_{1,m+1}^T\mathcal{V}_{1,j})^2+\text{trace}(\widehat{\mathcal{V}}_{2,m+1}^T\mathcal{V}_{2,j})^2+\text{trace}(\widehat{\mathcal{V}}_{3,m+1}^T\mathcal{V}_{3,j})^2},\quad j=1,\ldots,m$$ \;
							
							Set $\widehat{\mathcal{V}}_{i,m+1}=\widehat{\mathcal{V}}_{1,m+1}-\sum_{j=1}^m(\underline T_m)_{j,m}\mathcal{V}_{i,j}$, for $i=1,2,3$\;
						}
						
						Compute $(\underline T_m)_{m+1,m}=\sqrt{\|\widehat{\mathcal{V}}_{1,m+1}\|_F^2+\|\widehat{\mathcal{V}}_{2,m+1}\|_F^2+\|\widehat{\mathcal{V}}_{3,m+1}\|_F^2}$ \;
						
						Set $\mathcal{V}_{i,m+1}=\widehat{\mathcal{V}}_{i,m+1}/(\underline T_m)_{m+1,m}$, $i=1,2,3$\;
						
						Solve $y_m=\argmin_{y\in\mathbb{R}^m}\|\underline T_my-\beta e_1\|$\;
						
						Compute the residual norm $\|r_m\|$\;\label{alg:rescomp}
						
						\If{$\|r_m\|\leq\varepsilon\beta$}{
							Go to \ref{alg:lastline}\;
						}
					}
					Set $\delta\Theta_m=\sum_{j=1}^m\mathcal{V}_{1,j}(e_i^Ty_m)$, $\delta\Lambda_m=\sum_{j=1}^m\mathcal{V}_{2,j}(e_i^Ty_m)$, and $\delta X_m=\sum_{j=1}^m\mathcal{V}_{3,j}(e_i^Ty_m)$
					\label{alg:lastline}
					\caption{Matrix-oriented GMRES for~\eqref{eq:saddle-point}\label{matGMRES}}
				\end{algorithm}

				\setcounter{AlgoLine}{0}
				\begin{algorithm}[t]
					\SetKwInOut{Input}{input}\SetKwInOut{Output}{output}
					\Input{$B,Q_1,\ldots,Q_N\in\mathbb{R}^{s\times s},$ $R_0,\ldots,R_N\in\mathbb{R}^{p\times p}$, $H_0,\ldots,H_N\in\mathbb{R}^{p\times s}$, $M_1,\ldots,M_N\in\mathbb{R}^{s\times s}$ $[b_0,c_1,\ldots,c_N]\in\mathbb{R}^{s\times (N+1)}$, $[d_0,\ldots,d_N]\in\mathbb{R}^{p\times (N+1)}$, $m_{\max}$, $\varepsilon>0$}
					\Output{$\delta X_m\in\mathbb{R}^{s\times (N+1)}$ approximate solution to~\eqref{eq:CGproblem}}
					\BlankLine
					Set $\mathcal R_{0}= W_{0}=[Bb_0,Q_1c_1,\ldots,Q_Nc_N]+[H_0^TR_0^{-1}d_0-M_1^TH_1^TR_1^{-1}d_1,\ldots,H_{N-1}^TR_{N-1}^{-1}d_{N-1}-M_N^TH_N^TR_N^{-1}d_N,H_N^TR_N^{-1}d_N]$, $\delta X_0=0$, and
					compute $\rho_0=\|\mathcal R_{0}\|_F^2$ \;
					
					\For{$m=1, 2,\dots,$ till $m_{\max}$}{
						Set 
						{\small \begin{align*}
								\mathcal{W}_{m}=&[B^{-1}W_{m-1}e_1-M_1^TQ_1^{-1}(W_{m-1}e_2-M_1W_{m-1}e_1),  \\  &Q_1^{-1}(W_{m-1}e_2-M_1W_{m-1}e_1)-M_2^TQ_2^{-1}(W_{m-1}e_3-M_2W_{m-1}e_2),\ldots
								,Q_N^{-1}(W_{m-1}e_{N+1}-M_NW_{m-1}e_{N})]\\
								&+[H_0^TR_0^{-1}H_0W_{m-1}e_1,\ldots,H_N^TR_N^{-1}H_NW_{m-1}e_{N+1}]
							\end{align*}
						}\;
						
						$\alpha_m=\rho_{m-1}/\text{trace}(\mathcal{W}_{m}^TW_{m-1})$\;
						
						$\delta X_m=\delta X_{m-1}+\alpha_m W_{m-1}$\;
						
						$\mathcal{R}_m=\mathcal{R}_{m-1}-\alpha_k\mathcal{W}_m$\;
						
						$\rho_m=\|\mathcal{R}_m\|_F^2$\;
						
						\If{$\sqrt{\rho_m}\leq\varepsilon\rho_0$}{
							Return $\delta X_m$\;
						}
						
						$\beta_m=\rho_m/\rho_{m-1}$\;
						
						$W_m=\mathcal{R}_m+\beta_mW_{m-1}$\;
						
					}
					
					\caption{Matrix-oriented CG for~\eqref{eq:CGproblem}\label{matCG}}
				\end{algorithm}
				\bibliographystyle{siam}
				
				\bibliography{/bibSchur}

			\end{document}